\documentclass[11pt]{amsart}

\usepackage{latexsym,rawfonts}
\usepackage{amsfonts,amssymb}
\usepackage{amsmath,amsthm}

\usepackage[plainpages=false]{hyperref}
\usepackage{graphicx}
\usepackage{color}
\usepackage[table]{xcolor}
\usepackage{longtable}

\newcommand{\koo}{\mathcal K_0^n}
\newcommand{\ko}{\mathcal K^n}

\newcommand{\rnnn}{\mathbb R^n}
\newcommand{\ry}{\mathbb R}

\newcommand{\sn}{ {\mathbb{S}^{n-1}}}

\newcommand{\psum}{{+_{\negthinspace\kern-2pt p}}\,}
\newcommand{\qsum}[1]{{+_{\negthinspace\kern-2pt #1}}\,}
\newcommand{\dpsum}{{\tilde+_{\negthinspace\kern-1pt p}}\,}
\newcommand{\dqsum}[1]{{\tilde+_{\negthinspace\kern-1pt #1}}\,}
\newcommand{\lsub}[1]{\hskip -1.5pt\lower.5ex\hbox{$_{#1}$}}

\numberwithin{equation}{section}

\newtheorem{theo}{Theorem}[section]
\newtheorem{coro}[theo]{Corollary}
\newtheorem{lem}[theo]{Lemma}
\newtheorem{prop}[theo]{Proposition}

\newtheorem{rem}[theo]{Remark} \theoremstyle{definition}

\begin{document}

\title{The $L_{p}$-Brunn-Minkowski inequalities\\
 for variational functionals with $0\leq p<1$
}
\author[J. Hu]{Jinrong Hu}
\address{Institut f\"{u}r Diskrete Mathematik und Geometrie, Technische Universit\"{a}t Wien, Wiedner Hauptstrasse 8-10, 1040 Wien, Austria
 }
\email{jinrong.hu@tuwien.ac.at}

\begin{abstract}
The infinitesimal forms of the $L_{p}$-Brunn-Minkowski inequalities for variational functionals, such as the $q$-capacity, the torsional rigidity, and the first eigenvalue of the Laplace operator, are investigated for $p \geq 0$. These formulations yield  Poincar\'{e}-type inequalities related to these functionals. As an application, the $L_{p}$-Brunn-Minkowski inequalities for torsional rigidity with $0 \leq p < 1$ are confirmed for small smooth perturbations of the unit ball.
\end{abstract}
\keywords{$L_{p}$-Brunn-Minkowski inequality, Variational functionals}

\subjclass[2010]{52A20, 26D10}

\thanks{This work is supported by the Austrian Science Fund (FWF): Project P36545.}

\maketitle

\baselineskip18pt

\parskip3pt

\section{Introduction}

The classical Minkowski problem of prescribing the surface area measure and the Brunn-Minkowski inequality for volume are two foundational results in the Brunn-Minkowski theory. The former was originally proposed and addressed by Minkowski himself \cite{M897,M903}. This problem has seen significant progress through a series of papers (see, e.g., ~\cite{A42,A39,B87,FJ38}), which primarily studied the existence, regularity and uniqueness. The latter states that if $\Omega_{1}, \Omega_{2}\in\ko$  and $t\in [0,1]$, then
\begin{equation}\label{VE}
V((1-t)\Omega_{1}+t\Omega_{2})^{1/n}\geq (1-t)V(\Omega_{1})^{1/n}+t V(\Omega_{2})^{1/n},
\end{equation}
with equality if and only if $\Omega_{1}$ and $\Omega_{2}$ are homothetic. It can be expressed in an equivalent form
\begin{equation}\label{VE2}
V((1-t)\Omega_{1}+t\Omega_{2})\geq V(\Omega_{1})^{1-t}V(\Omega_{2})^{t},
\end{equation}
and for $t\in (0,1)$, there is equality if and only if $\Omega_{1}$ and $\Omega_{2}$ are translates. For standard references on the proof of \eqref{VE}, see \cite{G06,S14}. \eqref{VE} yields the uniqueness result for the classical Minkowski problem. In this beautiful survey, Gardner \cite{Ga02} illustrated that \eqref{VE} has many applications in \emph{Geometry and Analysis}. For instance,  the classical isoperimetric inequality can be deduced by using \eqref{VE}. Besides, by means of \eqref{VE}, Colesanti \cite{CA08} obtained Poincar\'{e} type inequalities on the boundary of convex bodies. The Poincar\'{e} and Brunn-Minkowski inequalities on the boundary of weighted Riemannian manifolds were derived by Kolesnikov-Milman \cite{KM18}.

As an extension of the Brunn-Minkowski theory, the $L_p$-Brunn-Minkowski theory has been actively investigated since the 1990s but dates back to the 1950s. In 1960s, Firey \cite{F62} extended the Minkowski combination of convex bodies to the $L_{p}$-Minkowski sum $(1-t)\cdot \Omega_{1} +_{p}t \cdot \Omega_{2} $ for $\Omega_{1}, \Omega_{2}\in \koo$, $p>1$, $t\in [0,1]$ and established the following $L_{p}$-Brunn-Minkowski inequality when $p> 1$:
\begin{equation}\label{}
V((1-t)\cdot  \Omega_{1}+_pt \cdot \Omega_{2})^{\frac{p}{n}}\geq (1-t)V( \Omega_{1})^{\frac{p}{n}}+t V(\Omega_{2})^{\frac{p}{n}},
\end{equation}
with equality if and only if $\Omega_{1}$ and $\Omega_{2}$ are dilates. The $L_{p}$ Minkowski problem of prescribing the $L_{p}$ surface area measure was first introduced and solved by Lutwak ~\cite{L93}. The case $p=1$ is the classical Minkowski problem, and the case $p=0$ is the log-Minkowski problem first solved by B\"{o}r\"{o}czky-Lutwak-Yang-Zhang \cite{BLZ13}. Building upon Lutwak's seminal work, the $L_{p}$ Minkowski problem has been the breeding ground with many meaningful outputs. However, the uniqueness result of the $L_{p}$ Minkowski problem for $p\in [0,1)$ is not completely solved; the reason behind this is that whether the $L_{p}$-Brunn-Minkowski inequality for $p\in[0,1)$ holds is  open at large. In this regard, B\"{o}r\"{o}czky-Lutwak-Yang-Zhang \cite{BLZ12}  proposed the following conjectured $L_{p}$-Brunn-Minkowski inequality for $p\in [0,1)$.

{\bf The conjectured $L_{p}$-Brunn-Minkowski inequality for $p\in [0,1)$.} The $L_{p}$-Brunn-Minkowski inequality holds for all origin-symmetric convex bodies $\Omega_{1}$, $\Omega_{2}$ in $\rnnn$:
\begin{equation}\label{CJ1}
V((1-t)\cdot \Omega_{1}+_pt \cdot \Omega_{2})^{\frac{p}{n}}\geq (1-t)V(\Omega_{1})^{\frac{p}{n}}+t V(\Omega_{2})^{\frac{p}{n}}, \ \forall t\in[0,1].
\end{equation}
In the case $p=0$, called the log-Brunn-Minkowski conjecture, \eqref{CJ1} is interpreted in the limiting sense as
\begin{equation}\label{CJ2}
V((1-t)\cdot \Omega_{1}+_0t \cdot \Omega_{2})\geq V(\Omega_{1})^{1-t}V(\Omega_{2})^{t}, \ \forall t\in[0,1].
\end{equation}

Note that \eqref{CJ2} is stronger than \eqref{VE2}, and the $L_{p}$-Brunn-Minkowski inequality is equivalent to the related $L_{p}$-Minkowski inequality. The above-conjectured inequalities were established for plane convex bodies in \cite{BLZ12}. Kolesnikov-Milman \cite{KM22} made important progress towards \eqref{CJ1} and \eqref{CJ2} by establishing the local versions of these inequalities.  Chen-Huang-Li-Liu \cite{CLHL20} extended the local results of Kolesnikov-Milman to the global version for $p$ close to 1. Recently, Ivaki-Milman \cite{IM24} obtained the even $L_{p}$-Minkowski inequalities and established the uniqueness in the even $L_{p}$-Minkowski problems for $p<1$, under a curvature pinching condition. See also \cite{CLM17,GN17,HI24,Iva23,IM24-b}.

Several notable studies have explored extensions and analogues of the aforementioned Minkowski type problems and Brunn-Minkowski type inequalities associated with various geometric invariants within the \emph{Calculus of Variations}. These include investigations into the Newtonian capacity, the torsional rigidity, the first
eigenvalue of the Laplacian and their generalizations, see, e.g., \cite{CA05,CF10,FZH20,Hu24,HL21,HSX18,HYZ18,J91,J96} and the references therein. In this paper, we aim to contribute to this research area. Recall that within the context of a compact convex set $\Omega \subset \mathbb{R}^{n}$ ($n\geq3$), the Newtonian capacity $C(\Omega)$ of $\Omega$ is defined as
\begin{equation}\label{cdf1*}
C(\Omega)=\inf\left\{\int_{\rnnn}|\nabla U|^{2}dX: U\in C^{\infty}_{c}({\rnnn})\ {\rm and} \ U\geq 1\ {\rm on}\ \Omega \right\}.
\end{equation}

From the perspective of analysis, by \cite{L77},  there exists a unique function $U\in W^{1,2}_{0}({\rnnn} \backslash \bar{\Omega})$ such that
\begin{equation}\label{cdf2*}
C(\Omega)=\int_{{\rnnn} \backslash \bar{\Omega}
}|\nabla U|^{2}dX,
\end{equation}
where the associated equilibrium potential $U$ satisfies the Euler-Lagrange equation,
\begin{equation*}\label{capequ*2}
\left\{
\begin{array}{lr}
\Delta U(X)=0, & X\in {\rnnn} \backslash \bar{\Omega}, \\
U(X)=1,  & X\in  \partial \Omega ,\\
\lim_{|X|\rightarrow \infty}U(X)=0.
\end{array}\right.
\end{equation*}
Jerison \cite{J96} established the Hadamard variational formula of $C(\cdot)$, which is given as
\begin{equation*}\label{torhadma}
\frac{d}{dt}C(\Omega+t\Omega_{1})\Big|_{t=0}=\int_{\sn}h_{\Omega_{1}}(\xi){d}\mu^{cap}(\Omega,\xi),
\end{equation*}
here the capacitary measure $\mu^{cap}(\Omega,\eta)$ is defined by
\begin{equation}\label{torme2}
\mu^{cap}(\Omega,\eta)=\int_{\nu^{-1}_{\Omega}(\eta)}|\nabla U(X)|^{2}{d}\mathcal{H}^{n-1}(X)=\int_{\eta}|\nabla U(\nu^{-1}_{\Omega}(\xi))|^{2}{d}S(\Omega,\xi)
\end{equation}
for every Borel subset $\eta\subset {\sn}$, $S(\Omega,\cdot)$ is the surface area measure of $\Omega$. Now, one sees that $\nabla U$ has finite non-tangential limit $\mathcal{H}^{n-1}$ a.e. on $\partial \Omega$ by using the estimates of harmonic functions proved by Dahlberg \cite{D77} and $|\nabla U|\in L^{2}(\partial \Omega, \mathcal{H}^{n-1})$, which imply that \eqref{torme2} is well-defined a.e. on the unit sphere ${\sn}$.

Jerison \cite{J96} proposed the Minkowski problem characterizing the capacitary measure and dealt with the existence and regularity of the solution via the variational method, which was firstly developed by Aleksandrov ~\cite{A42,A39}. Uniqueness was settled by Caffarelli-Jerison-Lieb in \cite{C96}. The  associated Brunn-Minkowski inequality is stated as follows.

The Brunn-Minkowski inequality for $C$ : given $\Omega_{1}, \Omega_{2}\in \ko$ and $t \in [0,1]$, then
\begin{equation}\label{CP}
C((1-t)\Omega_{1}+t \Omega_{2})^{1/(n-2)}\geq (1-t)C(\Omega_{1})^{1/(n-2)}+t C(\Omega_{2})^{1/(n-2)},
\end{equation}
with equality if and only if $\Omega_{1}$ and $\Omega_{2}$ are homothetic. A proof of \eqref{CP} was given by Borell \cite{B83}; subsequently, the equality case was characterized by Caffarelli-Jerison-Lieb \cite{C96}.

Based on Jerison's work, Colesanti-Nystr\"{o}m-Salani-Xiao-Yang-Zhang \cite{C15} studied the $q$-capacity for $1<q<n$,
\begin{equation}\label{cdf1*}
C_{q}(\Omega)=\inf\left\{\int_{\rnnn}|\nabla U|^{q}dX: U\in C^{\infty}_{c}({\rnnn})\ {\rm and} \ U\geq 1\ {\rm on}\ \Omega \right\}.
\end{equation}
Here $C:=C_{2}$. The associated $q$-equilibrium potential $U$ satisfies the boundary value problem (see e.g., \cite{L77}):
\begin{equation}\label{capequ*}
\left\{
\begin{array}{lr}
\Delta_{q} U(X)=0, & X\in {\rnnn} \backslash \bar{\Omega}, \\
U(X)=1,  & X\in  \partial \Omega ,\\
\lim_{|X|\rightarrow \infty}U(X)=0,
\end{array}\right.
\end{equation}
where $\Delta_{q}U:={\rm div}(|\nabla U|^{q-2}\nabla U)$ is the $q$-Laplace operator. In \cite{C15}, the authors established the Hadamard variational formula of $C_{q}(\cdot)$ as
\begin{equation*}\label{torhadma}
\frac{d}{dt}C_{q}(\Omega+t\Omega_{1})\Big|_{t=0}=(q-1)\int_{\sn}h_{\Omega_{1}}(\xi){d}\mu^{cap_{q}}(\Omega,\xi),
\end{equation*}
where the $q$-capacitary measure $\mu^{cap_{q}}(\Omega,\cdot)$ is defined as
\begin{equation*}\label{tormes2}
\mu^{cap_{q}}(\Omega,\eta)=\int_{\nu^{-1}_{\Omega}(\eta)}|\nabla U(X)|^{q}{d}\mathcal{H}^{n-1}(X)=\int_{\eta}|\nabla U(\nu^{-1}_{\Omega}(\xi))|^{q}{d}S(\Omega,\xi)
\end{equation*}
for every Borel subset $\eta$ of ${\sn}$, and defined the mixed $q$-capacity $C_{q}(\Omega,\Omega_{1})$ as
\[
C_{q}(\Omega,\Omega_{1})=\frac{q-1}{n-q}\int_{\sn}h_{\Omega_{1}}(\xi)d\mu^{cap_{q}}(\Omega,\xi).
\]
The case $q=2$ is classical with landmark contributions in Jerison's work \cite{J96}. Here we write $\mu^{cap}:=\mu^{cap_{2}}$. In \cite{LN07,LN08},  Lewis-Nystr\"{o}m generalized Dahlberg's results \cite{D77} ($q=2$) to the case $1<q<\infty$, this tells us that $\nabla U$ has finite non-tangential limits $\mathcal{H}^{n-1}$ a.e. on $\partial \Omega$ and $|\nabla U|\in L^{q}(\partial\Omega, \mathcal{H}^{n-1})$.

The Minkowski problem of prescribing the $q$-capacitary measure for $1<q<n$ was first introduced in \cite{C15}. There the existence and regularity of the solution were studied when $1<q<2$. Later Akman-Gong-Hineman-Lewis-Vogel \cite{A6} investigated the existence's result to $1<q<n$ and replacing the $q$-capacity by a generalized nonlinear capacity. The uniqueness for $1<q<n$  was settled in \cite{C15} and the associated Brunn-Minkowski inequality is as follows.

 The Brunn-Minkowski inequality for $C_q$ \cite{CS03} due to Colesanti and Salani : given $\Omega_{1}, \Omega_{2}\in\ko$,  for $1<q<n$ and $t \in [0,1]$, then
\begin{equation}\label{}
C_{q}((1-t)\Omega_{1}+t \Omega_{2})^{1/(n-q)}\geq (1-t)C_{q}(\Omega_{1})^{1/(n-q)}+t C_{q}(\Omega_{2})^{1/(n-q)},
\end{equation}
with equality if and only if $\Omega_{1}$ and $\Omega_{2}$ are homothetic.

Later, Zou-Xiong \cite{ZX14} extended  Colesanti-Nystr\"{o}m-Salani-Xiao-Yang-Zhang's work \cite{C15}, they established the $L_{p}$-Hadamard variational formula of $C_{q}$ for $1< p < \infty$ and $1<q<n$ as
 \[
 \frac{d C_{q}(\Omega+_{p}t\cdot \Omega_{1})}{dt}\Bigg|_{t=0}=\frac{q-1}{p}\int_{\sn}h^{p}_{\Omega_{1}}(\xi)d\mu^{cap_{q}}_{p}(\Omega,\xi),
 \]
where the $(p,q)$-capacitary measure is defined as
\begin{equation}\label{pmea1a}
\mu^{cap_{q}}_{p}(\Omega,\eta)=\int_{\nu^{-1}_{\Omega}(\eta)}h_{\Omega}(\nu_{\Omega}(X))^{1-p}{d}\mu^{cap_{q}}(\Omega,\eta)=\int_{\eta}h_{\Omega}(\xi)^{1-p}{d}\mu^{cap_{q}}(\Omega,\eta),
\end{equation}
and defined the mixed $(p,q)$-capacity $C_{p,q}(\Omega,\Omega_{1})$ as
\[
C_{p,q}(\Omega,\Omega_{1})=\frac{q-1}{n-q}\int_{\sn}h^{p}_{\Omega_{1}}(\xi)d\mu^{cap_{q}}_{p}(\Omega,\xi).
\]
Here we write $\mu^{cap}_{p}:=\mu^{cap_{2}}_{p}$. The authors in \cite{ZX14} also settled the existence  of solutions to the  $L_{p}$ Minkowski problem characterizing the $q$-capacitary measure when $p>1$ and $1<q<n$, Xiong-Xiong-Xu \cite{Xiong19} complemented the case $0<p<1$ and $1<q<2$. While the uniqueness is only obtained in the case $p>1$ and $1<q<n$ in \cite{ZX14}, which  is associated with the following inequality:

  The $L_p$-Brunn-Minkowski inequality for $C_q$ \cite{ZX14} : given $\Omega_{1}, \Omega_{2}\in\koo$, for $1< p <\infty$, $1<q<n$ and $t \in [0,1]$, then
\begin{equation}\label{LPQ0}
C_{q}((1-t)\cdot \Omega_{1}+_{p}t \cdot \Omega_{2})^{p/(n-q)}\geq (1-t)C_{q}(\Omega_{1})^{p/(n-q)}+t C_{q}(\Omega_{2})^{p/(n-q)},
\end{equation}
with equality if and only if $\Omega_{1}$ and  $\Omega_{2}$ are dilates.

Very recently, there have been some results on the existence of a solution to the log-Minkowski problem prescribing the cone $q$-capacitary measure in the case $p=0$ in \eqref{pmea1a}. The cone $q$-capacitary measure $G^{cap_{q}}(\Omega,\cdot)$ of a convex body $\Omega$ is defined for a Borel set $\eta\subset \sn$ by
\begin{equation}\label{CTC}
G^{cap_{q}}(\Omega,\eta)=\frac{q-1}{n-q}\int_{X\in \nu^{-1}_{\Omega}(\eta)}h_{\Omega}(\nu_{\Omega}(X))d\mu^{cap_{q}}(\Omega,\eta).
\end{equation}
Here we write $G^{cap}:=G^{cap_{2}}$. Xiong-Xiong \cite{XX22} solved this problem for polytopes. Note that the uniqueness question is open due to lack of a log-Brunn-Minkowski inequality for $q$-capacity --- analogous to the conjectured log-Brunn-Minkowski inequality.

 The torsional rigidity (abbreviated as torsion later) of $\Omega$ is defined as
\begin{equation}\label{cdf1*}
\frac{1}{T(\Omega)}=\inf\left\{\frac{\int_{\Omega}|\nabla U|^{2}{d}X}{(\int_{ \Omega}|U|{d}X)^{2}}:\ U\in W^{1,2}_{0}(\Omega)\ , \int_{\Omega}|U|{d}X> 0\right\}.
\end{equation}
For problem \eqref{cdf1*}, from \cite{CA05}, we know that there exists a unique function $U$ such that
\begin{equation*}\label{cdf2*}
T(\Omega)=\int_{\Omega} U dX,
\end{equation*}
where  $U$ is the solution of the boundary-value problem
\begin{equation}\label{Tor2}
\left\{
\begin{array}{lr}
\Delta U(X)= -1, & X\in \Omega, \\
U(X)=0,  & X\in  \partial \Omega.
\end{array}\right.
\end{equation}

For two arbitrary convex bodies $\Omega,\Omega_{1}$ in ${\rnnn}$, in analogy with the variational formula of volume (see, e.g., ~\cite{S14}),  Colesanti-Fimiani~\cite{CF10} obtained the variational formula of torsion,
\begin{equation}\label{capha*}
\frac{d}{dt}T(\Omega+t\Omega_{1})\Big|_{t=0}=\int_{\partial \Omega}h_{\Omega_{1}}(\nu_{\Omega}(X))|\nabla U(X)|^{2}{d}\mathcal{H}^{n-1}(X),
\end{equation}
where the torsional measure $\mu^{tor}(\Omega,\eta)$ is defined on the unit sphere ${\sn}$ by
\begin{equation}\label{capme*}
\mu^{tor}(\Omega,\eta)=\int_{\nu^{-1}_{\Omega}(\eta)}|\nabla U(X)|^{2}{d}\mathcal{H}^{n-1}(X)=\int_{\eta}|\nabla U(\nu^{-1}_{\Omega}(\xi))|^{2}dS(\Omega, \xi)
\end{equation}
for every Borel subset $\eta$ of ${\sn}$. In addition, they also defined the mixed torsion $T(\Omega,\Omega_{1})$ as
\[
T(\Omega,\Omega_{1})=\frac{1}{n+2}\int_{\sn}h_{\Omega_{1}}(\xi)d\mu^{tor}(\Omega,\xi).
\]
 Moreover, the
 estimates on harmonic function established by Dahlberg \cite{D77} illustrate that $\nabla U$ has finite non-tangential limit at $\mathcal{H}^{n-1}$-a.e. on $\partial \Omega$ with $|\nabla U| \in L^{2}(\partial \Omega, \mathcal{H}^{n-1})$, which is not limited by the smooth condition on $\Omega$. Hence, $\eqref{capme*}$ is well-defined a.e. on the unit sphere ${\sn}$.

The Minkowski problem prescribing the torsional measure was first posed by  Colesanti-Fimiani \cite{CF10}. They proved the existence and uniqueness of the solution up to translations via the variational method. The uniqueness is associated with the following Brunn-Minkowski inequality:

The Brunn-Minkowski inequality for $T$: given $\Omega_{1}, \Omega_{2}\in\ko$, for $t\in [0,1]$, then
\begin{equation}\label{TdE}
T((1-t)\Omega_{1}+t\Omega_{2})^{1/(n+2)}\geq (1-t)T(\Omega_{1})^{1/(n+2)}+tT(\Omega_{2})^{1/(n+2)}.
\end{equation}
The above inequality was first gained by Borell \cite{BC85}, and Colesanti \cite{CA05} proved that equality occurs in \eqref{TdE} if and only if $\Omega_{1}$ is homothetic to $\Omega_{2}$.

Similar to the $L_{p}$ surface area measure (see, e.g., ~\cite{S14}),  Chen-Dai \cite{Chen20} defined the $L_{p}$ torsional measure $\mu^{tor}_{p}(\Omega,\cdot)$ for $p\geq 1$ by
\begin{equation}\label{pmea1}
\mu^{tor}_{p}(\Omega,\eta)=\int_{\nu^{-1}_{\Omega}(\eta)}h_{\Omega}(\nu_{\Omega}(X))^{1-p}{d}\mu^{tor}(\Omega,\eta)=\int_{\eta}h_{\Omega}(\xi)^{1-p}{d}\mu^{tor}(\Omega,\eta)
\end{equation}
for every Borel subset $\eta$ of ${\sn}$. They also defined the $L_{p}$ mixed torsion $T_{p}(\Omega,\Omega_{1})$ as
\[
T_{p}(\Omega,\Omega_{1})=\frac{1}{n+2}\int_{\sn}h^{p}_{\Omega_{1}}(\xi)d\mu^{tor}_{p}(\Omega,\xi).
\]

The $L_{p}$ Minkowski problem prescribing the $L_{p}$ torsional measure was first introduced in \cite{Chen20}. The authors proved the existence and uniqueness of the solution when $p>1$. The uniqueness corresponds to the following inequality:

  The $L_p$-Brunn-Minkowski inequality for $T$ \cite{Chen20}: given $\Omega_{1}, \Omega_{2}\in\koo$, for $1< p <\infty$, and $t \in [0,1]$, then
\begin{equation}\label{LPQI}
T((1-t)\cdot \Omega_{1}+_{p}t \cdot \Omega_{2})^{p/(n+2)}\geq (1-t)T(\Omega_{1})^{p/(n+2)}+t T(\Omega_{2})^{p/(n+2)},
\end{equation}
with equality if and only if $\Omega_{1}$ and  $\Omega_{2}$ are dilates. For the case $0<p<1$, Hu-Liu \cite{HL21} gave a sufficient condition for the existence of a solution to the $L_{p}$ torsional Minkowski problem, while the uniqueness for $0<p<1$ is unknown since there is no related $L_{p}$-Brunn-Minkowski inequality. Specially, if $p=0$, the $L_{p}$ torsional measure of a convex body $\Omega$ corresponds to the
 cone torsional measure $G^{tor}(\Omega,\cdot)$, which is defined for a Borel set $\eta\subset \sn$ by
\begin{equation}\label{CTC}
G^{tor}(\Omega,\eta)=\frac{1}{n+2}\int_{X\in \nu^{-1}_{\Omega}(\eta)}h_{\Omega}(\nu_{\Omega}(X))d\mu^{tor}(\Omega,\eta).
\end{equation}
The Minkowski problem prescribing the cone torsional measure was recently studied by  \cite{LU23} with the evenness assumption and by \cite{Hu24} without the even assumption. The uniqueness of the torsion log-Minkowski problem is also challenging. The key could be to prove a log-Brunn-Minkowski inequality for torsion.

The first eigenvalue of the Laplace operator (abbreviated as first eigenvalue later) $\lambda(\Omega)$ of $\Omega$ is defined as follows
\begin{equation}\label{TAU}
\lambda(\Omega)=\inf \left\{\frac{\int_{\Omega}|\nabla U|^{2}dX}{\int_{\Omega}U^{2} dX}: U\in W^{1,2}_{0}(\Omega), \int_{\Omega}U^{2}dx>0  \right\}.
\end{equation}
Sakaguchi \cite{S88} proved that the infimum is attained for some function $U\in W^{1,2}_{0}(\Omega)$, satisfying
\begin{equation}\label{MA}
\left\{
\begin{array}{lr}
\Delta U(X)= -\lambda U, & X\in \Omega, \\
U(X)=0,  & X\in  \partial \Omega.
\end{array}\right.
\end{equation}
If we restrict the minimizer $U$ with $\int_{\Omega}U^{2} dX=1$, then
\[
\lambda(\Omega)=\int_{\Omega}|\nabla U|^{2}dX.
\]

Jerison \cite{JD96} established the Hadamard variational formula of $\lambda(\cdot)$, which is given as
\begin{equation}\label{torhadma}
\frac{d}{dt}\lambda(\Omega+t\Omega_{1})\Big|_{t=0}=-\int_{\sn}h_{\Omega_{1}}(\xi){d}\mu^{eig}(\Omega,\xi),
\end{equation}
here the eigenvalue measure $\mu^{eig}(\Omega,\eta)$ is defined by
\begin{equation}\label{tormes22}
\mu^{eig}(\Omega,\eta)=\int_{\nu^{-1}_{\Omega}(\eta)}|\nabla U(X)|^{2}{d}\mathcal{H}^{n-1}(X)=\int_{\eta}|\nabla U(\nu^{-1}_{\Omega}(\xi))|^{2}{d}S(\Omega,\xi)
\end{equation}
for every Borel subset $\eta\subset {\sn}$. Based on that, the mixed first eigenvalue  $\lambda(\Omega,\Omega_{1})$ is defined as
\[
\lambda(\Omega,\Omega_{1})=\frac{1}{2}\int_{\sn}h_{\Omega_{1}}(\xi)d\mu^{eig}(\Omega,\xi),
\]
 and   for a Borel set $\eta\subset \sn$, define the cone eigenvalue measure as
\begin{equation}\label{CTC}
G^{eig}(\Omega,\eta)=\frac{1}{2}\int_{X\in \nu^{-1}_{\Omega}(\eta)}h_{\Omega}(\nu_{\Omega}(X))d\mu^{eig}(\Omega,\eta).
\end{equation}

We remark that the Minkowski problem for the eigenvalue measure was also studied in \cite{JD96}. The uniqueness of this problem is related to the following inequality:

The Brunn-Minkowski inequality for $\lambda$ : For $\Omega_{1}, \Omega_{2}\in\ko$ and $t \in [0,1]$, then
\begin{equation}\label{Cta}
\lambda((1-t)\Omega_{1}+t \Omega_{2})^{-1/2}\geq (1-t)\lambda(\Omega_{1})^{-1/2}+t \lambda(\Omega_{2})^{-1/2},
\end{equation}
with equality if and only if $\Omega_{1}$ and $\Omega_{2}$ are homothetic. A proof of \eqref{Cta} was given by Brascamp-Lieb \cite{BL76}, another proof was shown by Borell \cite{B000}, and subsequently, the equality case was characterized by Colesanti \cite{CA05}.

As far as we know, there are very few results on the  $L_{p}$-Brunn-Minkowski inequalities for the above variational functionals with $0\leq p <1$. Very recently, Crasta-Fragal\`a \cite{CF24} made some progress; they proved that if the above cone variational measure is absolutely continuous with constant density and if the body is origin-centered, then the underlying body is a ball. In addition, in conjunction with the Saint-Venant inequality between volume and torsion, they obtained the following inequality for torsion:
\begin{equation*}\label{LTM}
\int_{\sn}\log\frac{h_{\Omega}}{h_{B^{n}}}d\bar{G}^{tor}(B^{n},\cdot)\geq \frac{1}{n+2}\log\frac{T(\Omega)}{T(B^{n})}
\end{equation*}
for a centrally symmetric convex body $\Omega$ and a unit ball $B^{n}$, where $d\bar{G}^{tor}(\Omega,\cdot)=\frac{d G^{tor}(\Omega,\cdot)}{T(\Omega)}$.

In this paper, the infinitesimal formulations on the $L_{p}$-Brunn-Minkowski inequalities of the above variational functionals are derived for $p\geq 0$. As applications,  the local $L_{p}$-Brunn-Minkowski inequalities for torsion with $ 0\leq p<1$ in small enough $C^{\infty}$-perturbations of the unit ball can be derived.

\begin{theo}[The $L_{p}$-Brunn-Minkowski inequalities for torsion with $0< p<1$]\label{main5}
Suppose $n\geq 2$. Let $\varphi\in C^{\infty}(\sn)$ be even. Then there exists a sufficiently  small $\varepsilon_{0}>0$ such that for every $s_{1},s_{2}\in [-\varepsilon_{0},\varepsilon_{0}]$ and for every $t\in [0,1]$,  there is
\begin{equation}\label{LPN}
T(t \cdot \Omega_{1}+_{p}(1-t)\cdot \Omega_{2})^{\frac{p}{n+2}}\geq t T(\Omega_{1})^{\frac{p}{n+2}}+(1-t)T(\Omega_{2})^{\frac{p}{n+2}},
\end{equation}
where $\Omega_{1}$ is the convex body with the support function $h_{1}=(1+s_{1}\varphi)^{\frac{1}{p}}$ and $\Omega_{2}$ is the convex body with the support function $h_{2}=(1+s_{2}\varphi)^{\frac{1}{p}}$.
\end{theo}
\begin{theo}[The log-Brunn-Minkowski inequality for torsion]\label{main6}
Suppose $n\geq 2$. Let $\phi\in C^{\infty}(\sn)$ be even. Then there exists a sufficiently small $\varepsilon_{0}>0$ such that for every $s_{1},s_{2}\in [-\varepsilon_{0},\varepsilon_{0}]$ and for every $t\in [0,1]$, there is
\begin{equation}\label{LPN2}
T(t\cdot \Omega_{1}+_{0}(1-t)\cdot \Omega_{2})\geq T(\Omega_{1})^{t}T(\Omega_{2})^{1-t},
\end{equation}
where $\Omega_{1}$ is the convex body with the support function $h_{1}=e^{ s_{1}\phi}$ and $\Omega_{2}$ is the convex body with the support function $h_{2}=e^{s_{2}\phi}$.
\end{theo}

\begin{rem}
In fact,  by the monotonically increasing properties of torsion and the $L_{p}$ (Minkowski) sum of convex bodies in $\rnnn$ with $p\geq 0$, one sees that the log-Brunn-Minkowski inequality for torsion leads to the $L_p$-Brunn-Minkowski inequalities for torsion with $0<p<\infty$.
 It is very interesting to strengthen the results of Theorems \ref{main5} and \ref{main6}, and investigate the $L_{p}$-Brunn-Minkowski inequality of the $q$-capacity and first eigenvalue for $0\leq p<1$.

\end{rem}

The organization of this paper is as follows: Sec. \ref{Sec2} includes some preliminaries on convex body and variational functionals. In Sec. \ref{Sec3}, we obtain the first and second order variational formulas of variational functionals under the $L_{p}$ sum. In Sec. \ref{Sec4}, we give a class of (sharp) Poincar\'{e} type inequalities for variational functionals coming from their related $L_{p}$-Brunn-Minkowski inequalities. Theorems \ref{main5} and \ref{main6} shall be proved in Sec. \ref{Sec5}.

\section{Preliminaries}
\label{Sec2}
This section recalls some basics on convex body and variational functionals.
\subsection{Basics of convex bodies}
There are many standard references about the theory of convex bodies, for instance, see Gardner \cite{G06} and Schneider \cite{S14}.

Write ${\rnnn}$ for the $n$-dimensional Euclidean space.   The unit ball in $\rnnn$ is denoted by  $B^n$. We write $\omega_{n}$ for the $n$-dimensional volume of $B^n$. For $Y, Z\in {\rnnn}$, $ Y\cdot Z $ denotes the standard inner product. For vectors $X\in{\rnnn}$, $|X|=\sqrt{ X\cdot X}$ is the Euclidean norm.  Let ${\sn}$ be the unit sphere, and $C({\sn})$ be the set of continuous functions defined on ${\sn}$. The set of even functions in $C^{\infty}(\sn)$ is denoted by $C^{\infty}_{e}(\sn)$. A convex body is a compact convex set of ${\rnnn}$ with a non-empty interior. The set of all convex bodies in $\rnnn$ is denoted by $\ko$. The set of all convex bodies that contain the origin in the interior in $\rnnn$ is denoted by $\koo$. Let $\mathcal{K}^{\infty}_{+,e}$ denote the class of origin-symmetric convex bodies with $C^{\infty}$ smooth boundary and strictly positive Gauss curvature.

If $\Omega\in \koo$, for $\xi\in{\sn}$, the support function of $\Omega$ (with respect to the origin) is defined by
\[
h_{\Omega}(\xi)=\max\{ \xi\cdot Y:Y \in \Omega\}.
\]
We may extend this definition to a homogeneous function of degree one in $\rnnn\backslash \{0\}$ by the equation $h_{\Omega} (\xi)=|\xi|h_{\Omega}\left(\frac{\xi}{|\xi|}\right)$.

The $L_{p}$ sum of  $\Omega_{1},\Omega_{2}\in \koo$ for $a,b>0$ and $p\neq 0$ is defined as (see \cite{F62})
\[
a\cdot \Omega_{1}+_{p} b\cdot \Omega_{2}= \bigcap_{\xi\in \sn}\left\{X\in \rnnn:X\cdot \xi\leq (ah_{\Omega_{1}}(\xi)^{p}+bh_{\Omega_{2}}(\xi)^{p})^{\frac{1}{p}}\right\},
\]
and the $L_{0}$ sum (log-Minkowski sum) of $\Omega_{1},\Omega_{2}\in \koo$ for $a,b>0$ is defined as
\[
a\cdot \Omega_{1}+_{0}b\cdot \Omega_{2}=\bigcap_{\xi\in \sn}\left\{X\in \rnnn:X\cdot\xi\leq h_{\Omega_{1}}(\xi)^{a}h_{\Omega_{2}}(\xi)^{b} \right\}.
\]

 Given a convex body $\Omega$ in $\rnnn$, for $\mathcal{H}^{n-1}$ almost all $X\in \partial \Omega$, the unit outer normal of $\Omega$ at $X$ is unique. In such case, we use $\nu_{\Omega}$ to denote the Gauss map that takes $X\in \partial \Omega$ to its unique unit outer normal.
  Meanwhile, for $\omega\subset {\sn}$, the inverse Gauss map $\nu_{\Omega}$ is expressed as
\begin{equation*}
\nu^{-1}_{\Omega}(\omega)=\{X\in \partial \Omega:  \nu_{\Omega}(X) {\rm \ is \ defined \ and }\ \nu_{\Omega}(X)\in \omega\}.
\end{equation*}

Let $h^{*}:=h_{\Omega^{*}}$ denote the support function of $\Omega^{*}$, where  $\Omega^{*}$ is the polar body of $\Omega$ defined by
\[
\Omega^{*}=\{ X\in \rnnn,X\cdot Y\leq 1 \ for \ all \ Y\in \Omega  \}.
\]
Notice that the Gauss map can be defined on $\partial \Omega$ as
\[
\nu_{\Omega}=\frac{\nabla h^{*}}{|\nabla h^{*}|}.
\]
Due to $h_{\Omega}(\nabla h^{*}(X))=1$ for $X\in \partial \Omega$, it follows that
\[
h_{\Omega}(\nu_{\Omega}(X))=\frac{1}{|\nabla h^{*}|}.
\]

In particular, for a convex body $\Omega$ of class $C^{\infty}_{+}$, i.e., its boundary is $C^{\infty}$-smooth and is of positive Gauss curvature, for simplicity in the subsequence, we abbreviate $\nu^{-1}_{\Omega}$ as $F$, and $h:=h_{\Omega}$. Then the support function of $\Omega$ can be written as
\begin{equation}\label{hhom}
h(\xi)=\xi\cdot F(\xi)=\nu_{\Omega}(X)\cdot X, \ {\rm where} \ \xi\in {\sn}, \ \nu_{\Omega}(X)=\xi \ {\rm and} \ X\in \partial \Omega.
\end{equation}
 Let $\{e_{1},e_{2},\ldots, e_{n-1}\}$ be a local orthonormal frame on ${\sn}$, $h_{i}$ be the first order covariant derivatives of $h_{\Omega}(\cdot)$  with respect to a local orthonormal frame on ${\sn}$. Differentiating \eqref{hhom} with respect to $e_{i}$ , we get
\[
h_{i}=e_{i}\cdot F(\xi)+\xi\cdot F_{i}(\xi).
\]
Since $F_{i}$ is tangent to $ \partial \Omega$ at $F(x)$, we obtain
\begin{equation}\label{Fi}
h_{i}=e_{i}\cdot F(\xi).
\end{equation}
Combining \eqref{hhom} and \eqref{Fi}, we have (see also \cite[p. 97]{U91})
\begin{equation}\label{Fdef}
F(\xi)=\sum_{i} h_{i}e_{i}+h\xi=\nabla_{\sn}h+h\xi.
\end{equation}
Here $\nabla_{\sn}$ is the (standard) spherical gradient. On the other hand, since we can extend $h_{\Omega}(\cdot)$ to $\rnnn$ as a 1-homogeneous function $h(\cdot)$, then restrict the gradient of $h(\cdot)$ on $\sn$, it yields that (see for example, \cite{CY76})
\begin{equation}\label{hf}
\nabla h(\xi)=F(\xi), \ \forall \xi\in{\sn},
\end{equation}
where $\nabla$ is the gradient operator in $\rnnn$. Let $h_{ij}$ be the second-order covariant derivatives of $h_{\Omega}$ with respect to a local orthonormal frame on ${\sn}$. Then, applying \eqref{Fdef} and \eqref{hf}, we have (see, e.g., \cite[p. 382]{J91})
\begin{equation}\label{hgra}
\nabla h(\xi)=\sum_{i}h_{i}e_{i}+h\xi, \quad F_{i}(\xi)=\sum_{j}(h_{ij}+h\delta_{ij})e_{j}.
\end{equation}
Also, the {\rm reverse Weingarten map} of $\Omega$ at $\xi$ is given by
\[
h_{ij}+h\delta_{ij}=D(\nu^{-1}_{\Omega}(\xi)) \quad {\rm on} \ \sn.
\]

We define the matrix
\[
Q(\tilde{h};\xi)=(\tilde{h}_{ij}(\xi)+\tilde{h}(\xi)\delta_{ij})_{i,j=1,\ldots,n-1}, \ \xi\in \sn,
\]
and the set
\begin{equation*}\label{S}
\mathcal{S}=\{\tilde{h}\in C^{\infty}({\sn}):\tilde{h}_{ij}+\tilde{h}\delta_{ij}>0 \  {\rm on} \ \sn \}.
\end{equation*}

The second fundamental form ${\rm II}_{\partial \Omega}$ at $X\in \partial \Omega$ is defined as ${\rm II}_{X}(Y,Z)=\nabla_{Y}\nu_{\Omega}\cdot Z$ with $Y,Z\in T_{X}\partial \Omega$. We write ${\rm II}^{-1}_{\partial \Omega}$ for the inverse $\Pi_{\partial \Omega}$. From the definition of the Weingarten map, one sees that the second fundamental form
of $\partial \Omega$ is equivalent to Weingarten map, i.e., $\Pi_{\partial \Omega}=D \nu_{\Omega}$.

The Gauss curvature of $\partial\Omega$, $\kappa$, i.e., the Jacobian determinant of Gauss map of $\partial\Omega$, is revealed as
\begin{equation*}
\kappa=\frac{1}{\det(h_{ij}+h\delta_{ij})}.
\end{equation*}

\subsection{Basics of variational functionals} For the reader's convenience, we collect some facts about the properties of capacity, torsion and first eigenvalue.

 Now, according to \cite{D77,J91,J96}, we first list some properties corresponding to the solution $U$ of \eqref{capequ*}, as shown below.

For any $ X\in\partial \Omega$, $0<b<1$, the non-tangential cone is defined as
\begin{equation*}\label{capm2}
\Gamma(X)=\left\{Y\in \rnnn \backslash
 \bar{\Omega}:dist(Y,\partial \Omega)> b|X-Y|\right\}.
\end{equation*}

\begin{lem}\label{Nonfin}
Let $\Omega\in \koo$ and $U$ be the solution of \eqref{capequ*} in $\rnnn \backslash
 \bar{\Omega}$. Then, the non-tangential limit
\begin{equation*}\label{capm3}
\nabla U(X)=\lim_{Y \rightarrow X, \ Y \in \Gamma(X)}\nabla U(Y)
\end{equation*}
exists for $\mathcal{H}^{n-1}$ almost all $X\in\partial \Omega$. Furthermore, for $\mathcal{H}^{n-1}$ almost all $X\in \partial \Omega$,
\begin{equation*}
\nabla U(X)=-|\nabla U(X)|\nu_{\Omega}(X)\quad{\rm and}\quad |\nabla U|\in L^{q}(\partial \Omega,  \mathcal{H}^{n-1}).
\end{equation*}
\end{lem}
With respect to $q$-capacity $C_{q}$ for $1<q<n$, there are the following properties.
\begin{lem}\cite{C15}\label{T1}
(i) It is positively homogenous of order $(n-q)$, i.e.,
\begin{equation*}\label{torhom}
C_{q}(m\Omega_{0})=m^{n-q}C_{q}(\Omega_{0}),\ m> 0.
\end{equation*}

(ii) It is translation invariant. That is
\begin{equation*}\label{tranin}
C_{q}(\Omega_{0} +x_{0})=C_{q}(\Omega_{0}),\ \forall x_{0}\in {\rnnn}.
\end{equation*}

(iii) It is monotonically increasing, i.e., $C_{q}(\Omega_{1})\leq C_{q}(\Omega_{2})$ if $\Omega_{1}\subset \Omega_{2}$.
\end{lem}

 Let $\Omega\in \koo$ be of class $C^{\infty}_{+}$, and $U$ be the solution of \eqref{capequ*} in $\rnnn \backslash \bar{\Omega}$. By a direct calculation in conjunction with Lemma \ref{Nonfin} and the divergence theorem, for $1<q<n$, \cite{C15} showed that
\begin{equation}\label{capdef2}
C_{q}(\Omega)=\frac{q-1}{n-q}\int_{\partial \Omega}|\nabla U(X)|^{q}(X\cdot \nu_{\Omega}(X)){d}{\mathcal{H}^{n-1}}(X).
\end{equation}
Since for every $f\in C({\sn})$,
\begin{equation*}\label{arsphere}
\int_{{\sn}}f(\xi){d}{S_{\Omega}(\xi)}=\int_{\partial \Omega}f(\nu_{\Omega}(X)){d}{\mathcal{H}^{n-1}}(X).
\end{equation*}
\eqref{capdef2} is equivalent to
\begin{align}\label{captatget2}
C_{q}(h):=C_{q}(\Omega)&=\frac{q-1}{n-q}\int_{{\sn}}|\nabla U(F(\xi))|^{q}h(\xi){d}{S(\Omega,\xi)}\notag\\
&=\frac{q-1}{n-q}\int_{{\sn}}|\nabla U(F(\xi))|^{q}h(\xi)\det(h_{ij}(\xi)+h(\xi)\delta_{ij}){d}{\xi},
\end{align}

Let $\Omega\in \koo$ be of class $C^{\infty}_{+}$, and $U$ be the solution of \eqref{Tor2} in $\Omega$, then $T(\Omega)$ can be expressed as (for more details,  see \cite{CA05,CF10})
\begin{align}\label{11T}
T(h):=T(\Omega)&=\frac{1}{n+2}\int_{{\sn}}|\nabla U(F(\xi))|^{2}h(\xi){d}{S(\Omega,\xi)}\notag\\
&=\frac{1}{n+2}\int_{{\sn}}|\nabla U(F(\xi))|^{2}h(\xi)\det(h_{ij}(\xi)+h(\xi)\delta_{ij}){d}{\xi}.
\end{align}
The torsion is positively homogeneous of order $n+2$ and is translation invariant and is monotonically increasing.

Let $\Omega\in \koo$ be of class $C^{\infty}_{+}$, and $U$ is the solution of \eqref{MA} in $\Omega$, $\lambda(\Omega)$ can be written as (for more details,  see \cite{CA05, JD96})
\begin{align}\label{112T}
\lambda(h):=\lambda(\Omega)&=\frac{1}{2}\int_{{\sn}}|\nabla U(F(\xi))|^{2}h(\xi){d}{S(\Omega,\xi)}\notag\\
&=\frac{1}{2}\int_{{\sn}}|\nabla U(F(\xi))|^{2}h(\xi)\det(h_{ij}(\xi)+h(\xi)\delta_{ij}){d}{\xi}.
\end{align}
The first eigenvalue is homogeneous of order $-2$ and is monotonically decreasing.

\section{ The first and second variational formulas of variational functionals under $L_{p}$ sum}
\label{Sec3}

In this section, we calculate the first and second variational formulas of variational functionals under $L_p$ sum.
\subsection{$q$-capacity}

We first list the following results proved  in  \cite[Lemma 3.3]{C15} or \cite[Lemma 2.5]{J96}. For the reader's convenience, we give the proof process.

\begin{lem}\label{UPO} Suppose $1<q<n$.  Let $\Omega_{t}\in \koo$ be of class $C^{\infty}_{+}$ in ${\rnnn}$ with $|t|< \varepsilon$ for $\varepsilon>0$, and $U(\cdot,t)$ be the solution of \eqref{capequ*} in $\rnnn \backslash \bar{\Omega}_{t}$, then

(i)$(\nabla^{2} U(F(\xi,t),t)e_{i})\cdot e_{j}=-\kappa |\nabla U(F(\xi,t),t)|c_{ij}(\xi,t)$;

(ii)$(\nabla^{2} U(F(\xi,t),t)\xi)\cdot \xi=(q-1)^{-1}\kappa |\nabla U(F(\xi,t),t)|\sum_{i} c_{ii}(\xi,t)$;

(iii)$(\nabla^{2} U(F(\xi,t),t)e_{i})\cdot \xi=-\kappa\sum_{j} |\nabla U(F(\xi,t),t)|_{j}c_{ij}(\xi,t)$,

where $c_{ij}(\cdot,t)$ is the cofactor matrix of $((h_{t})_{ij}+h_{t}\delta_{ij})$.
\end{lem}

\begin{proof}
First notice that $\Omega_{t}\in \koo$ is of class $C^{\infty}_{+}$, from an elementary barrier argument and boundary Schauder estimate for quasi-linear elliptic equations, see e.g., \cite{GT01}, one sees that partial derivatives of $U$ up to, and including, order two are pointwise well-defined on $\partial \Omega_{t}$. By virtue of this fact, the proof can be shown as below. Assume that $h_{t}(\xi)\in C^{\infty}({\sn})$ for each $t\in (-\varepsilon,\varepsilon)$ with $\varepsilon>0$ is the support function of $\Omega_{t}$ and let $h^{'}_{t}:=\frac{\partial h_{t}(\xi)}{\partial t}$. Then, recall \eqref{hf} and \eqref{hgra}, we get $F(\xi,t)=\sum_{i}(h_{t})_{i}e_{i}+h_{t}\xi$,  $\frac{\partial F(\xi,t)}{\partial t}:=\dot{F}(\xi,t)=\sum_{i}(h^{'}_{t})_{i}e_{i}+h^{'}_{t}\xi$, $F_{i}(\xi,t)=\sum_{j}w_{ij}e_{j}$ with $w_{ij}:=(h_{t})_{ij}+h_{t}\delta_{ij}$, and $F_{ij}(\xi,t)=\sum_{k}w_{ijk}e_{k}-w_{ij}\xi$, where $w_{ijk}$ are the covariant derivatives of $w_{ij}$. Note that, from Lemma \ref{Nonfin}, we know that $\nabla U(F(\xi,t),t)=-|\nabla U(F(\xi,t),t)|\xi$, thus we have
\begin{equation*}\label{U1Fdef}
 \nabla U \cdot F_{i}=0,
\end{equation*}
and
\begin{equation*}\label{U2Fdef}
 ((\nabla^{2}U)F_{j})\cdot F_{i}+\nabla U\cdot F_{ij}=0.
\end{equation*}
It follows that
\begin{equation}\label{U3Fdef}
\sum_{k,l} w_{ik}w_{jl}((\nabla^{2}U)e_{l})\cdot e_{k})+w_{ij}|\nabla U|=0.
\end{equation}
Multiplying both sides of ~\eqref{U3Fdef} by $c_{ip}c_{jq}$ and sum for $i, j$, we get
\begin{equation*}\label{U4Fdef}
 \sum_{k,l}\delta_{kp}\delta_{ql}((\nabla^{2}U e_{l})\cdot e_{k})=-\kappa \sum_{j}\delta_{jp}c_{jq}|\nabla U|.
\end{equation*}
Hence,
\begin{equation*}\label{U7Fdef}
(\nabla^{2}U e_{i})\cdot e_{j}=-\kappa c_{ij}|\nabla U|.
\end{equation*}
This proves $(i)$.

Second, recall that
\begin{equation*}\label{U8Fdef}
|\nabla U(F(\xi,t),t)|=-\nabla U(F(\xi,t),t)\cdot \xi.
\end{equation*}
Taking the covariant derivative of both sides, we have
\begin{align}\label{U9Fdef}
|\nabla U|_{j}&=-\nabla U\cdot e_{j}-(\nabla^{2}U)F_{j}\cdot \xi\notag\\
&=- \sum_{i} w_{ij}((\nabla ^{2}U)e_{i}\cdot \xi).
\end{align}
Multiplying both sides of \eqref{U9Fdef} by $c_{lj}$ and sum for $j$, and combining
\begin{equation*}\label{U10Fdef}
\sum_{j}c_{lj}w_{ij}=\delta_{li}\det((h_{t})_{ij}+h_{t}\delta_{ij}).
\end{equation*}
Then, we get
\begin{equation*}\label{U11Fdef}
\sum_{j}c_{ij}|\nabla U|_{j}=-\det((h_{t})_{ij}+h_{t}\delta_{ij})(\nabla^{2}U) e_{i}\cdot \xi.
\end{equation*}
Hence
\begin{equation*}\label{U12Fdef}
(\nabla^{2}U) e_{i}\cdot \xi=-\kappa \sum_{j}c_{ij}|\nabla U|_{j}.
\end{equation*}
This proves $(iii)$.

Last, using the $q$-Laplace equation for $U$, we get
\[
(q-2)((\nabla^{2}U)\xi \cdot \xi)+\Delta U=0.
\]
Employing $(i)$, we obtain
\begin{align*}\label{DIV}
(q-1)(\nabla^{2}U)\xi\cdot \xi&=(\nabla^{2} U)\xi\cdot \xi-\Delta U\notag\\
&=-\sum_{i}(\nabla^{2}U)e_{i}\cdot e_{i}=\kappa \sum_{i}c_{ii}|\nabla U|.
\end{align*}
Hence, the proof is completed.
\end{proof}
We have the following result along the same line of proof as in \cite[Lemma 3.1]{C15}.
\begin{lem}\label{ar}
Suppose $1<q<n$. Let $\Omega_{t}\in \koo$ be of class $C^{\infty}_{+}$ in ${\rnnn}$ with $|t|< \varepsilon$ for $\varepsilon>0$, and $U(\cdot,t)$ be the solution of \eqref{capequ*} in $\rnnn \backslash \bar{\Omega}_{t}$. Then for each fixed $X\in  {\rnnn} \backslash \bar{\Omega}_{t}$, the function $t\rightarrow U(X,t)$ is differentiable with respect to $t$. Let $\dot{U}(X,t)=\frac{\partial U}{\partial t}(X,t)$. The function $\dot{U}(\cdot,t): \rnnn \backslash \bar{\Omega}_{t}\rightarrow \ry$ can be extended to $\partial \Omega_{t}$ so that $\dot{U}(\cdot,t)\in C^{2}(\rnnn \backslash \Omega_{t})$. Moreover,
\begin{equation*}
\label{Ut}
\left\{
\begin{array}{lr}
\nabla\cdot ((q-2)|\nabla U|^{q-4}(\nabla U \cdot  \nabla \dot{U})\nabla U+|\nabla U|^{q-2}\nabla \dot{U})=0, & X\in {\rnnn} \backslash \bar{\Omega}_{t}, \\
\dot{U}(X,t)=|\nabla U(X,t)|h^{'}_{t},   & X\in  \partial \Omega_{t},
\end{array}\right.
\end{equation*}
where $h_{t}$ is the support function of $\Omega_{t}$, and $h^{'}_{t}$ denotes the derivative of $h_{t}$ with respect to $t$.
Furthermore, there exists $c=c(n,q)$ such that
\begin{equation*}\label{}
0\leq \dot{U}\leq c|X|^{\frac{q-n}{q-1}}\ as\ |X|\rightarrow \infty,
\end{equation*}
\begin{equation*}\label{qq}
0\leq |\nabla \dot{U}|\leq c|X|^{\frac{1-n}{q-1}}\ as\ |X|\rightarrow \infty.
\end{equation*}
In addition,
\begin{equation*}\label{}
\nabla \dot{U}(X):=\lim_{Y\rightarrow X, Y\in \Gamma (X)}\nabla \dot{U}(Y) \ exists  \ for \ every \ X\in \partial \Omega_{t},
\end{equation*}
and
\[
\int_{\partial \Omega_{t}}|\nabla U|^{q-1}|\nabla \dot{U}|d \mathcal{H}^{n-1}< \infty.
\]
\end{lem}

Inspired by \cite[Lemma 3.2]{C15}, the following lemma  can also be obtained.
\begin{lem}\label{arb}
Suppose $1<q<n$. Let $\Omega_{t}\in \koo$ be of class $C^{\infty}_{+}$ in ${\rnnn}$ with $|t|< \varepsilon$ for $\varepsilon>0$, and $U(\cdot,t)$ be the solution of \eqref{capequ*} in $\rnnn \backslash \bar{\Omega}_{t}$. Let $\Omega_{t,i}\in \koo$ be of class $C^{\infty}_{+}$ in ${\rnnn}$ for $i=1,2$, and $h_{t,i}$ be the support function of $\Omega_{t,i}$. Let $U_{i}(\cdot,t)$ be the solution of \eqref{capequ*} in $\rnnn \backslash \bar{\Omega}_{t,i}$.  Let $\dot{U}_{i}(X,t)=\frac{\partial U_{i}}{\partial t}(X,t)$ for all $X\in   {\rnnn} \backslash \Omega_{t}$. Then,
 \begin{equation}
\begin{split}\label{L1selfad}
&\int_{\partial \Omega_{t}}h^{'}_{t,1}(\nu_{\Omega_{t}}(X))|\nabla U(X,t)|^{q-1}(\nu_{\Omega_{t}}(X)\cdot \nabla \dot{U}_{2}(X,t)){d}{\mathcal{H}^{n-1}(X)}\\
&=\int_{\partial \Omega_{t}}h^{'}_{t,2}(\nu_{\Omega_{t}}(X))|\nabla U(X,t)|^{q-1}(\nu_{\Omega_{t}}(X)\cdot \nabla \dot{U}_{1}(X,t)){d}{\mathcal{H}^{n-1}(X)}.
\end{split}
\end{equation}
\end{lem}
Now, applying Lemmas \ref{ar}, \ref{arb}, we can get the first and second variational formulas of $C_{q}$ under $L_{p}$ sum.

\begin{lem}\label{fsd}
Suppose $p>0$ and $1<q<n$. Let $h\in \mathcal{S}$ be positive, $\phi^{p}\in C^{\infty}(\sn)$ and $\varepsilon>0$ such that $(h^p+t\phi^p)^{1/p} \in \mathcal{S}$ for every $t\in (-\varepsilon,\varepsilon)$. Set $h_{t}=(h^p+t\phi^p)^{1/p}$ and $f_{p,q}(t)=C_{q}(h_{t})$. Then
\begin{equation*}\label{}
f^{'}_{p,q}(0)=\frac{q-1}{p}\int_{\sn}\phi^{p}h^{1-p}\frac{ |\nabla U(F(\xi))|^{q}}{\kappa}d \xi,
\end{equation*}
and
\begin{equation*}
\begin{split}
\label{Up3}
f^{''}_{p,q}(0)&=\frac{(1-p)(q-1)}{p^2}\int_{\sn}h^{1-2p}\phi^{2p}|\nabla U(F(\xi))|^{q}\frac{1}{\kappa}d\xi\\
&\quad -\frac{q(q-1)}{p}\int_{\sn}h^{1-p}\phi^{p} \frac{|\nabla U(F(\xi))|^{q-1}}{\kappa}\nabla \dot{U}(F(\xi))\cdot \xi d \xi\\
&\quad-\frac{1}{p^{2}}\int_{\sn}\sum_{i}|\nabla U(F(\xi))|^{q}c_{ii}h^{2-2p}\phi^{2p}d \xi\\
&\quad +\frac{(q-1)}{p^2}\int_{\sn}\sum_{i,j}h^{1-p}\phi^{p}(c_{ij}|\nabla U(F(\xi))|^{q}(h^{1-p}\phi^p)_{i})_{j}d \xi,
\end{split}
\end{equation*}
where $\dot{U}$ satisfies the following equation
\begin{equation*}
\label{Ut}
\left\{
\begin{array}{lr}
\nabla\cdot ((q-2)|\nabla U|^{q-4}(\nabla U \cdot  \nabla \dot{U})\nabla U+|\nabla U|^{q-2}\nabla \dot{U})=0, & X\in {\rnnn} \backslash \bar{\Omega}, \\
\dot{U}(X)=\frac{1}{p}\phi^{p}(\nu_{\Omega}(X))h^{1-p}(\nu_{\Omega}(X))|\nabla U(X)|,   & X\in  \partial \Omega,\\
\lim_{|X|\rightarrow \infty}\dot{U}(X)=0.
\end{array}\right.
\end{equation*}
\end{lem}
\begin{proof}
 In view of Lemma \ref{ar}, we first conclude that $U(F(\xi,t),t)$ is differentiable with respect to $t$. Based on this fact and using \eqref{captatget2}, by a direct calculation, we have
\begin{equation}
\begin{split}
\label{capcalcu}
f^{'}_{p,q}(t)&=\frac{d}{dt}C_{q}(h_{t})\\
&=\frac{q-1}{n-q}\int_{{\sn}}(h^{'}_{t} G(h_{t})+h_{t}\frac{d}{dt}G(h_{t}))d \xi, \ \forall \xi\in {\sn},
\end{split}
\end{equation}
where
\begin{equation}\label{Fdef1*}
G(h_{t})=|\nabla  U(F(\xi,t),t)|^{q}\det((h_{t})_{ij}+h_{t}\delta_{ij})(\xi),
\end{equation}
 with
 \begin{equation*}\label{Fdef2}
 F(\xi)=\nu^{-1}_{\Omega}(\xi)=\nabla h(\xi)=\sum_{i}h_{i}(\xi)e_{i}(\xi)+h(\xi)\xi.
\end{equation*}
By \eqref{Fdef1*}, we have
 \begin{equation}
\begin{split}
\label{Gcal}
\frac{d}{dt}G(h_{t})&=\frac{d}{dt}[|\nabla  U(F(\xi,t),t)|^{q}\det((h_{t})_{ij}+h_{t}\delta_{ij})]\\
&=|\nabla U(F(\xi,t),t)|^{q}(c^{t}_{ij}((h^{'}_{t})_{ij}+(h^{'}_{t})\delta_{ij}))+\det((h_{t})_{ij}+h_{t}\delta_{ij})\frac{d}{dt}(|\nabla U(F(\xi,t),t)|^{q})\\
&=|\nabla U(F(\xi,t),t)|^{q}(c^{t}_{ij}((h^{'}_{t})_{ij}+(h^{'}_{t})\delta_{ij}))\\
&\quad+q|\nabla U(F(\xi,t),t)|^{q-1}\det((h_{t})_{ij}+h_{t}\delta_{ij})\frac{d}{dt}(|\nabla U(F(\xi,t),t)|),
\end{split}
\end{equation}
where $c^{t}_{ij}:=c_{ij}(\cdot,t)$ is the cofactor matrix of $((h_{t})_{ij}+h_{t}\delta_{ij})$.

 Since $|\nabla U(F(\xi,t),t)|=-\nabla U(F(\xi,t),t)\cdot \xi$, $F(\xi,t)=\sum_{i}(h_{t})_{i}e_{i}+h_{t}\xi$, and $\dot{F}(\xi,t)=\sum_{i}(h^{'}_{t})_{i}e_{i}+h^{'}_{t}\xi$, then we get
 \begin{equation}
\begin{split}
\label{GraU}
&\frac{d}{dt}(|\nabla U(F(\xi,t),t)|)\\
&=-\nabla^{2}U(F(\xi,t),t)\xi\cdot \dot{F}(\xi,t)-\nabla \dot{U}(F(\xi,t),t)\cdot \xi\\
&=-\nabla^{2}U(F(\xi,t),t))\xi\cdot (\sum_{i}(h^{'}_{t})_{i}e_{i}+h^{'}_{t}\xi)-\nabla \dot{U}(F(\xi,t),t)\cdot \xi.
\end{split}
\end{equation}
Applying Lemma \ref{UPO} into \eqref{GraU}, we have
 \begin{equation}
\begin{split}
\label{GraU2}
&\frac{d}{dt}(|\nabla U(F(\xi,t),t)|)\\
&=-\nabla^{2}U(F(\xi,t),t)\xi\cdot (\sum_{i}(h^{'}_{t})_{i}e_{i}+h^{'}_{t}\xi)-\nabla \dot{U}(F(\xi,t),t)\cdot \xi\\
&=\kappa \sum_{i,j}c^{t}_{ij}|\nabla U|_{j}(h^{'}_{t})_{i}-(q-1)^{-1}\kappa |\nabla U|\sum_{i}c^{t}_{ii}h^{'}_{t}-\nabla \dot{U}(F(\xi,t),t)\cdot \xi.
\end{split}
\end{equation}
Substituting \eqref{GraU2} into \eqref{Gcal}, we have
 \begin{equation}
\begin{split}
\label{Gcal2}
&\frac{d}{dt}G(h_{t})\\
&=|\nabla U(F(\xi,t),t)|^{q}\sum_{i,j}c^{t}_{ij}((h^{'}_{t})_{ij}+(h^{'}_{t})\delta_{ij})\\
&\quad+q|\nabla U(F(\xi,t),t)|^{q-1}\det((h_{t})_{ij}+h_{t}\delta_{ij})\left[\kappa \sum_{i,j}c^{t}_{ij}|\nabla U|_{j}(h^{'}_{t})_{i}\right.\\
&\left.\quad \quad \quad \quad \quad-(q-1)^{-1}\kappa |\nabla U(F(\xi,t),t)|\sum_{i}c^{t}_{ii}(h^{'}_{t})-\nabla \dot{U}(F(\xi,t),t)\cdot \xi\right]\\
&=-q\frac{|\nabla U(F(\xi,t),t)|^{q-1}}{\kappa}\nabla \dot{U}(F(\xi,t),t)\cdot \xi-\frac{q}{q-1}|\nabla U(F(\xi,t),t)|^{q}\sum_{i}c^{t}_{ii}(h^{'}_{t})\\
&\quad +q\sum_{i,j}c^{t}_{ij}|\nabla U(F(\xi,t),t)|^{q-1}|\nabla U(F(\xi,t),t)|_{j}(h^{'}_{t})_{i}+|\nabla U(F(\xi,t),t)|^{q}\sum_{i,j}c^{t}_{ij}((h^{'}_{t})_{ij}+(h^{'}_{t})\delta_{ij}).
\end{split}
\end{equation}
For simplifying \eqref{Gcal2}, we apply the conclusion that  $\Sigma_{j}c_{ijj}=0$ proved by \cite{CY76}, where $c_{ijl}$ is the covariant derivative tensor of $c_{ij}$.

We calculate
 \begin{equation}
\begin{split}
\label{Gcal3}
&\frac{d}{dt}G(h_{t})\\
&=-q\frac{|\nabla U(F(\xi,t),t)|^{q-1}}{\kappa}\nabla \dot{U}(F(\xi,t),t)\cdot \xi-(q-1)^{-1}|\nabla U(F(\xi,t),t)|^{q}\sum_{i}c^{t}_{ii}(h^{'}_{t})\\
&\quad +\sum_{i,j}(c^{t}_{ij}|\nabla U(F(\xi,t),t)|^{q}(h^{'}_{t})_{i})_{j}\\
&:=\mathcal{L}^{t}(h^{'}_{t}).
\end{split}
\end{equation}
Now we show that the operator $\mathcal{L}^{t}$ is self-adjoint on $L^{2}({\sn},d\xi)$ equipped with the standard scalar product. To see this, define for $j=1,2,3$,
 \begin{equation*}\label{L1}
\mathcal{L}^{t}_{1}(h^{'}_{t})=-q\frac{|\nabla U(F(\xi,t),t)|^{q-1}}{\kappa}\nabla \dot{U}(F(\xi,t),t)\cdot \xi,
\end{equation*}
\begin{equation*}\label{L2}
\mathcal{L}^{t}_{2}(h^{'}_{t})=-(q-1)^{-1}|\nabla U(F(\xi,t),t)|^{q}\sum_{i}c^{t}_{ii}(h^{'}_{t}),
\end{equation*}
and
\begin{equation*}\label{L3}
\mathcal{L}^{t}_{3}(h^{'}_{t})=\sum_{i,j}(c^{t}_{ij}|\nabla U(F(\xi,t),t)|^{q}(h^{'}_{t})_{i})_{j}.
\end{equation*}
We verify that  $\mathcal{L}^{t}_{1}, \mathcal{L}^{t}_{2}$ and $\mathcal{L}^{t}_{3}$ are self-adjoint on $L^{2}({\sn},d\xi)$, i.e., for $h^{'}_{t,1}(\xi), h^{'}_{t,2}(\xi)\in \mathcal{S}$:
\begin{equation}\label{selfad}
\int_{{\sn}}h^{'}_{t,1}\mathcal{L}^{t}_{j}(h^{'}_{t,2}){d}{\xi}=\int_{{\sn}}h^{'}_{t,2}\mathcal{L}^{t}_{j}(h^{'}_{t,1}){d}{\xi}, \quad j=1,2,3.
\end{equation}
Clearly, $\mathcal{L}^{t}_{2}$ is self-adjoint, $\mathcal{L}^{t}_{3}$ is self-adjoint by applying integration by parts, and Lemmas \ref{ar} and \ref{arb} indicate that $\mathcal{L}^{t}_{1}$ is self-adjoint. We conclude that, $\mathcal{L}^{t}$ is self-adjoint on $L^{2}({\sn},d\xi)$. In view of \eqref{Fdef1*}, $G(h_{t})$ is positively homogeneous of order $(n-q-1)$, for $s\geq 0$, set $G(h_{t}+sh_{t})=G((1+s)h_{t})=(1+s)^{n-q-1}G(h_{t})$, then we have
\begin{equation}\label{GVhom}
\frac{d}{ds}G((1+s)h_{t})\Big|_{s=0^{+}}=(n-q-1)G(h_{t})=\mathcal{L}^{t}(h_{t}).
\end{equation}
Applying \eqref{GVhom} into \eqref{capcalcu}, we obtain
\begin{equation}
\begin{split}
\label{Fcapcalcu}
\frac{d}{dt}f_{p,q}(t)&=\frac{q-1}{n-q}\int_{{\sn}}(h^{'}_{t} G(h_{t})+h_{t}\mathcal{L}^{t}(h^{'}_{t})){d}{\xi}\\
&=\frac{q-1}{n-q}\int_{{\sn}}(h^{'}_{t} G(h_{t})+h^{'}_{t}\mathcal{L}^{t}(h_{t})){d}{\xi}\\
&=\frac{q-1}{n-q}\int_{{\sn}}(h^{'}_{t} G(h_{t})+(n-q-1)(h^{'}_{t}) G(h_{t})){d}{\xi}\\
&=(q-1)\int_{{\sn}}h^{'}_{t} G(h_{t}){d}{\xi},
\end{split}
\end{equation}
namely,
\begin{align}\label{qwe}
f^{'}_{p,q}(t)&=(q-1)\int_{{\sn}}(h^{'}_{t})(\xi)|\nabla  U(F(\xi,t),t)|^{q}\det((h_{t})_{ij}(\xi)+h_{t}(\xi)\delta_{ij}){d}{\xi}.
\end{align}
Using \eqref{Fcapcalcu}, we get
\begin{equation}
\begin{split}
\label{qwe2}
f^{''}_{p,q}(t)=(q-1)\int_{\sn}(h^{''}_{t})G(h_{t})d\xi+(q-1)\int_{\sn}(h^{'}_{t})\mathcal{L}^{t} (h^{'}_{t}) d\xi.
\end{split}
\end{equation}
Then
\begin{align*}\label{Fcapcalcu2}
f^{'}_{p,q}(0)&=\frac{q-1}{p}\int_{{\sn}}h^{1-p}\phi^{p}|\nabla  U(F(\xi)|^{q}\det(h_{ij}(\xi)+h(\xi)\delta_{ij}){d}{\xi},
\end{align*}
and
\begin{equation*}
\begin{split}
\label{Up3}
f^{''}_{p,q}(0)&=\frac{(1-p)(q-1)}{p^2}\int_{\sn}h^{1-2p}\phi^{2p}|\nabla U(F(\xi))|^{q}\frac{1}{\kappa}d\xi\\
&\quad -\frac{q(q-1)}{p}\int_{\sn}h^{1-p}\phi^{p} \frac{|\nabla U(F(\xi))|^{q-1}}{\kappa}\nabla \dot{U}(F(\xi))\cdot \xi d \xi\\
&\quad - \frac{1}{p^{2}}\int_{\sn}\sum_{i}|\nabla U(F(\xi))|^{q}c_{ii}h^{2-2p}\phi^{2p}d \xi\\
&\quad +\frac{(q-1)}{p^2}\int_{\sn}\sum_{i,j}h^{1-p}\phi^{p}(c_{ij}|\nabla U(F(\xi))|^{q}(h^{1-p}\phi^p)_{i})_{j}d \xi.
\end{split}
\end{equation*}
Hence, Lemma \ref{fsd} is completed.
\end{proof}

\subsection{Torsion}

We calculate the first and second variational formulas of torsion under $L_p$ sum.
We first list the following results proved similarly in Lemma \ref{UPO}.

\begin{lem}\label{UPO8v} Let $\Omega_{t}\in \koo$ be of class $C^{\infty}_{+}$ in ${\rnnn}$ with $|t|< \varepsilon$ for $\varepsilon>0$, and $U(\cdot,t)$ be the solution of \eqref{Tor2} in $\Omega_{t}$, then

(i) $(\nabla^{2} U(F(\xi,t),t)e_{i})\cdot e_{j}=-\kappa |\nabla U(F(\xi,t),t)|c_{ij}(\xi,t)$;

(ii) $(\nabla^{2} U(F(\xi,t),t)\xi)\cdot \xi=\kappa |\nabla U(F(\xi,t),t)|\sum_{i} c_{ii}(\xi,t)-1$;

(iii) $(\nabla^{2} U(F(\xi,t),t)e_{i})\cdot \xi=-\kappa\sum_{j} |\nabla U(F(\xi,t),t)|_{j}c_{ij}(\xi,t)$.
\end{lem}

Building upon \eqref{Tor2}, we can get the following result.
\begin{lem}\label{arv}
 Let $\Omega_{t}\in \koo$ be of class $C^{\infty}_{+}$ in ${\rnnn}$ with $|t|< \varepsilon$ for $\varepsilon>0$, and $U(\cdot,t)$ be the solution of \eqref{Tor2} in $\Omega_{t}$. Then, for each fixed $X\in  \Omega_{t}$, the function $t\rightarrow U(X,t)$ is differentiable with respect to $t$. Let $\dot{U}(X,t)=\frac{\partial U}{\partial t}(X,t)$, then
\begin{equation*}
\label{Ut}
\left\{
\begin{array}{lr}
\Delta \dot{U}=0, & X\in \Omega_{t}, \\
\dot{U}(X,t)=|\nabla U(X,t)|h^{'}_{t},   & X\in  \partial \Omega_{t},
\end{array}\right.
\end{equation*}
where $h_{t}$ is the support function of $\Omega_{t}$, and $h^{'}_{t}$ denotes the derivative of $h_{t}$ with respect to $t$.
\end{lem}
\begin{proof}
 Following similar lines as \cite[Lemma 3.1]{C15} or \cite[Lemm B]{J91}. With the aid of the maximum principle of linear elliptic equation (see \cite[Chapter 3]{GT01}) with respect to \eqref{Tor2}, $U(\cdot,t)> 0$ in $\Omega_{t}$, together with the comparison principle of linear elliptic equation, the Schauder estimates of linear elliptic equation (see for instance, \cite[Theorem 3.7, Theorem 6.2]{GT01}), and using \cite[Lemma 2.3]{CF10} to obtain the upper bound of $\Big| \frac{U(\cdot,t+s)-U(\cdot,t)}{s}\Big|$ and the Schauder estimates of $\frac{U(\cdot,t+s)-U(\cdot,t)}{s}$ in the $C^{2,\alpha}$ norm space for $\alpha\in (0,1)$, whenever $t\in (-\varepsilon,\varepsilon)$ and $s\in (-\varepsilon,\varepsilon) \backslash \{0\} $ with $\varepsilon>0$. Then, applying the Arzel\`a-Ascoli Theorem and a standard diagonalization procedure, we obtain a sequence $\{s_{k}\}_{k\in \mathbb{N}}$, tending to $0$ as $k$ tends to infinity, and a function $\dot{U}(\cdot,t):\Omega_{t}\rightarrow \ry$, such that, as $k\rightarrow \infty$, one sees that $\frac{U(\cdot,t+s_{k})-U(\cdot,t)}{s_{k}}$ converges to $\dot{U}(\cdot,t)$ uniformly on compact sets of $\Omega_{t}$. Since $\Delta \left(\frac{U(\cdot,t+s)-U(\cdot,t)}{s}\right)=0$ in a compact subset of  $\Omega_{t}$, we have $\Delta \dot{U}(X,t)=0$ for $X\in \Omega_{t}$. Thus, the existence of the limit of $U(\cdot, t+s)$ as $s\rightarrow 0$, at least up to choosing a suitable sequence of $s$ in the interior of $\Omega_{t}$ is proved. So for each fixed $X\in \Omega_{t}$, the differentiability of the function $t\mapsto U(X,t)$ with respect to $t$ at $X$ has been derived. Furthermore, the function $\dot{U}(\cdot,t): \Omega_{t}\rightarrow \ry$ can be extended to $\partial \Omega_{t}$ so that $\dot{U}(\cdot,t)\in C^{2}(\bar{\Omega}_{t})$. Since $U(F(\xi,t),t)=0$ on $\partial \Omega_{t}$, take the derivative of both sides with respect to $t$,  there is
\begin{equation*}\label{Uderva4}
\dot{U}(F(\xi,t),t)+\nabla U(F(\xi,t),t)\cdot \dot{F}((\xi,t),t)=0.
\end{equation*}
It is further deduced as
\begin{align*}\label{Uderva5}
\dot{U}(F(\xi,t),t) &=-\nabla U(F(\xi,t),t)\cdot (\sum_{i}(h^{'}_{t})_{i}e_{i}+h^{'}_{t} \xi)\notag\\
&=|\nabla U(F(\xi,t),t)|\xi\cdot(\sum_{i}(h^{'}_{t})_{i}e_{i}+h^{'}_{t} \xi)\notag\\
&=|\nabla U(F(\xi,t),t)|h^{'}_{t}.
\end{align*}
Hence, the proof is completed.
\end{proof}
Based on the above lemmas, along similar lines as Lemma \ref{fsd},  the first and second derivatives of torsion under $L_{p}$ sum are shown as follows.
\begin{lem}\label{fsd5}
Suppose $p> 0$. Let $h\in \mathcal{S}$ be positive, $\phi^{p}\in C^{\infty}(\sn)$ and $\varepsilon>0$ such that $(h^p+t\phi^p)^{1/p} \in \mathcal{S}$ for every $t\in (-\varepsilon,\varepsilon)$. Set $h_{t}=(h^p+t\phi^p)^{1/p}$ and $g_{p}(t)=T(h_{t})$. Then
\begin{equation}\label{cz1}
g^{'}_{p}(0)=\frac{1}{p}\int_{\sn}\phi^{p}h^{1-p}\frac{ |\nabla U(F(\xi))|^{2}}{\kappa}d \xi,
\end{equation}
and
\begin{equation}
\begin{split}
\label{cz2}
g^{''}_{p}(0)&=\frac{(1-p)}{p^2}\int_{\sn}h^{1-2p}\phi^{2p}|\nabla U(F(\xi))|^{2}\frac{1}{\kappa}d\xi\\
&\quad -\frac{2}{p}\int_{\sn}h^{1-p}\phi^{p} \frac{|\nabla U(F(\xi))|}{\kappa}\nabla \dot{U}(F(\xi))\cdot \xi d \xi\\
&\quad-\frac{1}{p^{2}}\int_{\sn}\sum_{i}|\nabla U(F(\xi))|^{2}c_{ii}h^{2-2p}\phi^{2p}d \xi\\
&\quad +\frac{2}{p^{2}}\int_{\sn}h^{2-2p}\phi^{2p}|\nabla U(F(\xi))|\frac{1}{\kappa}d\xi\\ &\quad +\frac{1}{p^2}\int_{\sn}\sum_{i,j}h^{1-p}\phi^{p}(c_{ij}|\nabla U(F(\xi))|^{2}(h^{1-p}\phi^p)_{i})_{j}d \xi,
\end{split}
\end{equation}
where $\dot{U}$ satisfies the following equation
\begin{equation*}
\label{Ut}
\left\{
\begin{array}{lr}
\Delta \dot{U}=0, & X\in  \Omega, \\
\dot{U}(X)=\frac{1}{p}\phi^{p}(\nu_{\Omega}(X))h^{1-p}(\nu_{\Omega}(X))|\nabla U(X)|,   & X\in  \partial \Omega.
\end{array}\right.
\end{equation*}
\end{lem}

\subsection{First eigenvalue}

Our aim is to calculate the first and second variational formulas of the first eigenvalue of the Laplace operator under $L_p$ sum. We need some preparations.

\begin{lem}\label{UPO8s} Let $\Omega_{t}\in \koo$ be of class $C^{\infty}_{+}$ in ${\rnnn}$ with $|t|< \varepsilon$ for $\varepsilon>0$, and $U(\cdot,t)$ be the solution of \eqref{MA} in $\Omega_{t}$, then

(i)$(\nabla^{2} U(F(\xi,t),t)e_{i})\cdot e_{j}=-\kappa |\nabla U(F(\xi,t),t)|c_{ij}(\xi,t)$;

(ii)$(\nabla^{2} U(F(\xi,t),t)\xi)\cdot \xi=\kappa |\nabla U(F(\xi,t),t)|\sum_{i} c_{ii}(\xi,t)-\lambda U$;

(iii)$(\nabla^{2} U(F(\xi,t),t)e_{i})\cdot \xi=-\kappa\sum_{j} |\nabla U(F(\xi,t),t)|_{j}c_{ij}(\xi,t)$.
\end{lem}

Based on \eqref{MA}, we have the following result.
\begin{lem}\label{ars}
 Let $\Omega_{t}\in \koo$ be of class $C^{\infty}_{+}$ in ${\rnnn}$ with $|t|< \varepsilon$ for $\varepsilon>0$, and $U(\cdot,t)$ be the solution of \eqref{MA} in $\Omega_{t}$. Then, for each fixed $X\in  \Omega_{t}$, the function $t\rightarrow U(X,t)$ is differentiable with respect to $t$. Let $\dot{U}(X,t)=\frac{\partial U}{\partial t}(X,t)$, then
\begin{equation*}
\label{Ut}
\left\{
\begin{array}{lr}
\Delta \dot{U}=-\lambda \dot{U}, & X\in \Omega_{t}, \\
\dot{U}(X,t)=|\nabla U(X,t)|h^{'}_{t},   & X\in  \partial \Omega_{t},
\end{array}\right.
\end{equation*}
where $h_{t}$ is the support function of $\Omega_{t}$, and $h^{'}_{t}$ denotes the derivative of $h_{t}$ with respect to $t$.
\end{lem}

Analogously, based on above lemmas, the first and second derivatives of the first eigenvalue are obtained as follows.

\begin{lem}\label{fsdq}
Suppose $p> 0$. Let $h\in \mathcal{S}$ be positive, $\phi^{p}\in C^{\infty}(\sn)$ and $\varepsilon>0$ such that $(h^p+t\phi^p)^{1/p} \in \mathcal{S}$ for every $t\in (-\varepsilon,\varepsilon)$. Set $h_{t}=(h^p+t\phi^p)^{1/p}$ and $\Lambda_{p}(t)=\lambda(h_{t})$. Then
\begin{equation}\label{1p}
\Lambda^{'}_{p}(0)=-\frac{1}{p}\int_{\sn}\phi^{p}h^{1-p}\frac{ |\nabla U(F(\xi))|^{2}}{\kappa}d \xi,
\end{equation}
and
\begin{equation}
\begin{split}
\label{2p}
\Lambda^{''}_{p}(0)&=-\frac{(1-p)}{p^2}\int_{\sn}h^{1-2p}\phi^{2p}|\nabla U(F(\xi))|^{2}\frac{1}{\kappa}d\xi\\
&\quad +\frac{2}{p}\int_{\sn}h^{1-p}\phi^{p} \frac{|\nabla U(F(\xi))|}{\kappa}\nabla \dot{U}(F(\xi))\cdot \xi d \xi\\
&\quad+\frac{1}{p^{2}}\int_{\sn}\sum_{i}|\nabla U(F(\xi))|^{2}c_{ii}h^{2-2p}\phi^{2p}d \xi\\
&\quad -\frac{2}{p^{2}}\int_{\sn}\lambda Uh^{2-2p}\phi^{2p}|\nabla U(F(\xi))|\frac{1}{\kappa}d\xi\\ &\quad-\frac{1}{p^2}\int_{\sn}\sum_{i,j}h^{1-p}\phi^{p}(c_{ij}|\nabla U(F(\xi))|^{2}(h^{1-p}\phi^p)_{i})_{j}d \xi,
\end{split}
\end{equation}
where $\dot{U}$ satisfies the following equation
\begin{equation*}
\label{Ut}
\left\{
\begin{array}{lr}
\Delta \dot{U}=-\lambda \dot{U}, & X\in  \Omega, \\
\dot{U}(X)=\frac{1}{p}\phi^{p}(\nu_{\Omega}(X))h^{1-p}(\nu_{\Omega}(X))|\nabla U(X)|,   & X\in  \partial \Omega.
\end{array}\right.
\end{equation*}
\end{lem}

\section{Poincar\'{e} type inequalities for variational functionals}
\label{Sec4}

In the section, we will obtain some Poincar\'{e} type inequalities for variational functionals by using the $L_{p}$-Brunn-Minkowski inequalities for variational functions with $p\geq 1$.
\subsection{$q$-capacity}

Our first aim is to give the infinitesimal form of the $L_{p}$-Brunn-Minkowski inequality for $C_q$ for $1<q<n$.
 Some preparations are necessary.

  Let $h\in \mathcal{S}$ with $h>0$ be the support function of a convex body $\Omega\in \koo$ of class $C^{\infty}_{+}$, and let $\phi^{p}\in C^{\infty}(\sn)$, then there exists a sufficiently small $\varepsilon>0$ such that
\[
h_{t}:=(h^{p}+t\phi^{p})^{1/p}\in \mathcal{S}, \ \forall t\in [-\varepsilon,\varepsilon].
\]
For an interval $I:=[-\varepsilon,\varepsilon]$, we set the one-parameter family of convex bodies $\Omega_{t}\in \koo$ of class $C^{\infty}_{+}$:
\begin{equation}\label{phic}
\Pi(h,\phi,I)=\{ \Omega_{t}: h_{\Omega_{t}}= h_{t}= (h^{p}+t\phi^{p})^{1/p}, t\in I \}.
\end{equation}

Recall the $L_{p}$-Brunn-Minkowski inequality for $C_{q}$ with $p\geq 1$ \eqref{LPQ0}, then we have the following lemma.

\begin{lem}\label{LQ7}
\eqref{LPQ0} implies that for every one-parameter family $\Pi(h,\phi,I)$, with $h$, $\phi$, and $\Omega_{t}\in \Pi(h,\phi,I)$,
\begin{equation}\label{lpq768}
\frac{d^{2}}{dt^{2}}[C_{q}(\Omega_{t})]\Bigg|_{t=0}C_{q}(\Omega)\leq \frac{n-q-p}{n-q}\left( \frac{d}{dt}[C_{q}(\Omega_{t})]\Bigg|_{t=0}\right)^{2}.
\end{equation}
\end{lem}
\begin{proof}
Let $h\in \mathcal{S}$ with $h>0$ and $\phi^{p}\in C^{\infty}(\sn)$. In view that there exists a sufficiently small $\varepsilon>0$ such that $h_{t}:=(h^{p}+t\phi^{p})^{1/p}\in  \mathcal{S}$ is the support function a convex body $\Omega_{t}$ for all $t\in [-\varepsilon,\varepsilon]$. Then for $t\in [-\varepsilon,\varepsilon]$ and $s\in [0,1]$,
\[
h_{s\cdot \Omega_{t}+_{p} (1-s)\cdot \Omega}=(s h^{p}_{t}+(1-s) h^{p})^{1/p}.
\]
Since \eqref{LPQ0} holds, fix a small enough positive $t_{0}\in (0,\varepsilon]$,
\[
C_{q}(s\cdot \Omega_{t_{0}}+_{p} (1-s)\cdot \Omega)^{\frac{p}{n-q}}\geq s C_{q}(\Omega_{t_{0}})^{\frac{p}{n-q}}+(1-s)C_{q}(\Omega)^{\frac{p}{n-q}},
\]
this  implies that the second order derivative of $[0,1]\ni s\mapsto C_{q}(s\cdot \Omega_{t_{0}}+_{p} (1-s)\cdot \Omega)^{\frac{p}{n-q}}$ at $s=0$ is non-positive, yielding \eqref{lpq768}.

\end{proof}

Now we give the infinitesimal formulation of the $L_{p}$-Brunn-Minkowski inequality for $C_{q}$ as follows.
\begin{theo}\label{TP2p}
 Let $\Omega \in \koo$ be of class $C^{\infty}_{+}$ and $U$ be the solution of \eqref{capequ*} in $\rnnn \backslash \bar{\Omega}$. Let $p\geq1$ and $1<q<n$. Then for every function $\phi^{p}\in C^{\infty}(\sn)$, we have
\begin{equation}
\begin{split}
\label{aq6}
&\frac{(p+q-n)}{(n-q)C_{q}(\Omega)}\left(\int_{\sn}\phi^{p}h^{1-p}d \mu^{cap_{q}}\right)^{2}+ (1-p)(q-1)\int_{\sn}\phi^{2p}h^{1-2p}d \mu^{cap_{q}}\\
&\quad -q(q-1) \int_{\sn }\phi^{2p}h^{2-2p}\frac{(\nabla \dot{U}(F(\xi))\cdot \xi)}{\dot{U}(F(\xi))}d \mu^{cap_{q}}-\int_{\sn}{\rm tr}((Q^{-1}(h))\phi^{2p}h^{2-2p}d \mu^{cap_{q}}\\
&\leq (q-1)\int_{\sn}Q^{-1}(h)\nabla_{\sn} \left( \phi^{p}h^{1-p}\right)\cdot \nabla_{\sn} \left(\phi^{p}h^{1-p}\right)d \mu^{cap_{q}},
\end{split}
\end{equation}
where $\dot{U}$ satisfies the following equation
\begin{equation*}
\label{Ut}
\left\{
\begin{array}{lr}
\nabla\cdot ((q-2)|\nabla U|^{q-4}(\nabla U \cdot  \nabla \dot{U})\nabla U+|\nabla U|^{q-2}\nabla \dot{U})=0, & X\in {\rnnn} \backslash \bar{\Omega}, \\
\dot{U}(X)=\frac{1}{p}\phi^{p}(\nu_{\Omega}(X))h^{1-p}(\nu_{\Omega}(X))|\nabla U(X)|,   & X\in  \partial \Omega,\\
\lim_{|X|\rightarrow \infty}\dot{U}(X)=0.
\end{array}\right.
\end{equation*}
\end{theo}
\begin{proof}

 Assume first that $\phi^{p}\in C^{\infty}(\sn)$. Let $\varepsilon>0$ be such that $(h^{p}+t\phi^{p})^{1/p}\in \mathcal{S}$ for every $t\in (-\varepsilon,\varepsilon)$. Set  $f_{p,q}(t)=C_{q}((h^p+t\phi^p)^{1/p})$ and $\Upsilon_{p,q}(t)=f_{p,q}(t)^{p/(n-q)}$ for $t\in (-\varepsilon,\varepsilon)$. Recall Lemma \ref{LQ7}, we know that $\Upsilon_{p,q}$ is concave at $t=0$, then we get
\begin{equation*}\label{}
\Upsilon^{''}_{p,q}(0)=\frac{p(p+q-n)}{(n-q)^{2}}f_{p,q}(0)^{\frac{p}{n-q}-2}(f^{'}_{p,q}(0))^{2}+\frac{p}{n-q}f_{p,q}(0)^{\frac{p}{n-q}-1}f^{''}_{p,q}(0)\leq 0.
\end{equation*}
By using Lemma \ref{fsd}, we have
\begin{equation}
\begin{split}
\label{ert}
&\frac{(p+q-n)}{(n-q)C_{q}(\Omega)}\left(\int_{{\sn}}h^{1-p}\phi^{p}|\nabla  U(F(\xi))|^{q}\det(h_{ij}(\xi)+h(\xi)\delta_{ij}){d}{\xi}\right)^{2}\\
&\quad + (1-p)(q-1)\int_{\sn}h^{1-2p}\phi^{2p}|\nabla U(F(\xi))|^{q}\det(h_{ij}(\xi)+h(\xi)\delta_{ij}){d}{\xi}\\
&\quad -pq(q-1)\int_{\sn}h^{1-p}\phi^{p}\frac{|\nabla U(F(\xi))|^{q-1}}{\kappa} \nabla \dot{U}(F(\xi))\cdot \xi d \xi\\
&\quad -\int_{\sn}{\rm tr}(c_{ij})|\nabla U(F(\xi))|^{q}h^{2-2p}\phi^{2p}d \xi\\
&\leq -(q-1)\int_{\sn}\sum_{i,j}h^{1-p}\phi^{p}(c_{ij}|\nabla U(F(\xi))|^{q}(h^{1-p}\phi^p)_{i})_{j}d \xi.
\end{split}
\end{equation}
By an integration by parts,  there is
\begin{equation}
\begin{split}
\label{a160}
&\frac{(p+q-n)}{(n-q)C_{q}(\Omega)}\left(\int_{{\sn}}h^{1-p}\phi^{p}|\nabla  U(F(\xi)|^{q}\det(h_{ij}(\xi)+h(\xi)\delta_{ij}){d}{\xi}\right)^{2}\\
&\quad + (1-p)(q-1)\int_{\sn}h^{1-2p}\phi^{2p}|\nabla U(F(\xi))|^{q}\det(h_{ij}(\xi)+h(\xi)\delta_{ij}){d}{\xi}\\
&\quad -pq(q-1)\int_{\sn}h^{1-p}\phi^{p}|\nabla U(F(\xi))|^{q-1}\det(h_{ij}(\xi)+h(\xi)\delta_{ij})\nabla \dot{U}(F(\xi))\cdot \xi d \xi\\
&\quad -\int_{\sn}{\rm tr}(c_{ij})|\nabla U(F(\xi))|^{q}h^{2-2p}\phi^{2p}d \xi\\
&\leq (q-1)\int_{\sn}\sum_{i,j}(h^{1-p}\phi^{p})_{j}(c_{ij}|\nabla U(F(\xi))|^{q}(h^{1-p}\phi^p)_{i})d \xi.
\end{split}
\end{equation}
Note that
\begin{equation}\label{cijh0}
c_{ij}=\det(h_{ij}+h\delta_{ij})(h_{ij}+h\delta_{ij})^{-1}.
\end{equation}
Then \eqref{aq6} holds by substituting \eqref{cijh0} and the definition of $\mu^{cap_{q}}$ into \eqref{a160}. The proof is completed.

\end{proof}
 Now we prove the Poincar\'{e} type inequalities for $q$-capacity originating from the corrsponding $L_{p}$-Brunn-Minkowski inequalities for $p\geq 1$.
\begin{theo}\label{main1}
Let $\Omega \in \koo$ be of class $C^{\infty}_{+}$ and $U$ be the solution of \eqref{capequ*} in $\rnnn \backslash \bar{\Omega}$. Suppose $1\leq p<\infty$ and $1<q<n$. For every function $\psi^{p}\in C^{\infty}(\partial \Omega)$, then we obtain
\begin{equation}
\begin{split}
\label{aq90}
&\frac{(p+q-n)}{(n-q)C_{q}(\Omega)}\left(\int_{\partial \Omega}\frac{\psi^{p}}{|\nabla h^{*}|^{1-p}}|\nabla U|^{q}d \mathcal{H}^{n-1}(X)\right)^{2}\\
&\quad + (1-p)(q-1)\int_{\partial \Omega}\frac{\psi^{2p}}{|\nabla h^{*}|^{1-2p}}|\nabla U|^{q}d \mathcal{H}^{n-1}(X)\\
&\quad -pq(q-1) \int_{\partial \Omega }\frac{\psi^{p}}{|\nabla h^{*}|^{1-p}}(\nabla \dot{U}(X)\cdot \nu_{\Omega})|\nabla U|^{q-1}d \mathcal{H}^{n-1}(X)\\
&\quad -\int_{\partial \Omega}{\rm tr}(\Pi_{\partial\Omega})\frac{\psi^{2p}}{|\nabla h^{*}|^{2-2p}}|\nabla U|^{q}d \mathcal{H}^{n-1}(X)\\
&\leq (q-1)\int_{\partial \Omega}(\Pi^{-1}_{\partial \Omega})\nabla_{\partial \Omega} \left( \frac{\psi^{p}}{|\nabla h^{*}|^{1-p}}\right)\cdot \nabla_{\partial \Omega} \left( \frac{\psi^{p}}{|\nabla h^{*}|^{1-p}}\right)|\nabla U|^{q}d \mathcal{H}^{n-1}(X),
\end{split}
\end{equation}
where $\nabla_{\partial \Omega}$ denotes the induced Levi-Civita connection on the boundary $\partial \Omega$, $h^{*}$ is the support function of the polar body of $\Omega$, and $\dot{U}$ satisfies the following equation
\begin{equation*}
\label{Ut}
\left\{
\begin{array}{lr}
\nabla\cdot ((q-2)|\nabla U|^{q-4}(\nabla U \cdot  \nabla \dot{U})\nabla U+|\nabla U|^{q-2}\nabla \dot{U})=0, & X\in {\rnnn} \backslash \bar{\Omega}, \\
\dot{U}(X)=\frac{\psi^{p}}{p|\nabla h^{*}|^{1-p}}|\nabla U(X)|,   & X\in  \partial \Omega,\\
\lim_{|X|\rightarrow \infty}\dot{U}(X)=0.
\end{array}\right.
\end{equation*}
\end{theo}

\begin{proof}
Let $\psi^{p} \in C^{\infty}(\partial \Omega)$ and set $\phi(\xi)=\psi(F(\xi))$ for every $\xi\in \sn$, then $\phi^{p} \in C^{\infty}(\sn)$. By applying $\xi=\nu_{\Omega}(X)$, we get
 \begin{equation*}\label{}
\int_{\sn}\phi^{p}(\xi)\det(h_{ij}(\xi)+h(\xi)\delta_{ij})d \xi=\int_{\partial \Omega}\psi^{p} d \mathcal{H}^{n-1}(X).
\end{equation*}
Then
\begin{equation}
\begin{split}
\label{a2}
&\int_{{\sn}}h^{1-p}\phi^{p}|\nabla  U(F(\xi)|^{q}\det(h_{ij}(\xi)+h(\xi)\delta_{ij}){d}{\xi}\\
&=\int_{\partial \Omega}\frac{\psi^{p}}{|\nabla h^{*}|^{1-p}}|\nabla U|^{q}d \mathcal{H}^{n-1}(X),
\end{split}
\end{equation}
\begin{equation}
\begin{split}
\label{a22}
&\int_{{\sn}}h^{1-2p}\phi^{2p}|\nabla  U(F(\xi))|^{q}\det(h_{ij}(\xi)+h(\xi)\delta_{ij}){d}{\xi}\\
&=\int_{\partial \Omega}\frac{\psi^{2p}}{|\nabla h^{*}|^{1-2p}}|\nabla U|^{q}d \mathcal{H}^{n-1}(X),
\end{split}
\end{equation}
and
\begin{equation}
\begin{split}
\label{a3}
&\int_{\sn}h^{1-p}\phi^{p}|\nabla U(F(\xi))|^{q-1}\det(h_{ij}(\xi)+h(\xi)\delta_{ij})\nabla \dot{U}(F(\xi)))\cdot \xi d \xi\\
&= \int_{\partial \Omega }\frac{\psi^{p}}{|\nabla h^{*}|^{1-p}}(|\nabla U|^{q-1}\nabla \dot{U}\cdot \nu_{\Omega}) d \mathcal{H}^{n-1}(X).
\end{split}
\end{equation}
Since
\begin{equation}\label{aab}
c_{ij}=\det (h_{ij}+h\delta_{ij})(h_{ij}+h\delta_{ij})^{-1}.
\end{equation}
From \eqref{aab}, we get
\begin{equation}\label{a4}
\int_{\sn}{\rm tr}(c_{ij}(\xi))|\nabla U(F(\xi))|^{q}h^{2-2p}\phi^{2p}(\xi)d\xi=\int_{\partial \Omega}{\rm tr}(\Pi_{\partial\Omega})|\nabla U(X)|^{q}\frac{\psi^{2p}}{|\nabla h^{*}|^{2-2p}}d \mathcal{H}^{n-1}(X).
\end{equation}
Furthermore,
\begin{equation*}
\begin{split}
\label{a5}
&\sum_{i,j}c_{ij}(\xi)|\nabla U(F(\xi))|^{q}(h^{1-p}\phi^{p})_{i}(\xi)(h^{1-p}\phi^{p})_{j}(\xi)\\
&=\det(h_{ij}(\xi)+h(\xi)\delta_{ij})|\nabla U(F(\xi))|^{q}D(\nu^{-1}_{\Omega}(\xi))\nabla_{\partial \Omega} \left( \frac{\psi(\nu^{-1}_{\Omega}(\xi))^{p}}{|\nabla h^{*}(\nu^{-1}_{\Omega}(\xi))|^{1-p}}\right)\cdot \nabla_{\partial \Omega} \left( \frac{\psi(\nu^{-1}_{\Omega}(\xi))^{p}}{|\nabla h^{*}(\nu^{-1}_{\Omega}(\xi))|^{1-p}}\right).
\end{split}
\end{equation*}
Thus
\begin{equation}
\begin{split}
\label{a5}
&\int_{\sn}\sum_{i,j}c_{ij}(\xi)|\nabla U(F(\xi))|^{q}(h^{1-p}\phi^{p})_{i}(\xi)(h^{1-p}\phi^{p})_{j}(\xi)d\xi\\
&=\int_{\partial \Omega}|\nabla U(X)|^{q}\Pi^{-1}_{\partial\Omega}\nabla_{\partial \Omega} \left( \frac{\psi^{p}}{|\nabla h^{*}|^{1-p}}\right)\cdot \nabla_{\partial \Omega} \left( \frac{\psi^{p}}{|\nabla h^{*}|^{1-p}}\right)d \mathcal{H}^{n-1}(X).
\end{split}
\end{equation}
Now substituting \eqref{a2},\eqref{a22}, \eqref{a3}, \eqref{a4} and \eqref{a5} into Theorem \ref{TP2p},  we get Theorem \ref{main1} in the case $\psi^{p} \in C^{\infty}(\partial \Omega)$. The proof is completed.

\end{proof}
As a consequence of Theorem \ref{main1}, we obtain:
\begin{coro}
Suppose $1<q<n$. Let $\Omega \in \koo$ be of class $C^{\infty}_{+}$ and $U$ be the solution of \eqref{capequ*} in $\rnnn \backslash \bar{\Omega}$. For every function $\psi\in C^{\infty}(\partial \Omega)$,  if
\[
\int_{\partial \Omega}\psi|\nabla U|^{q}d \mathcal{H}^{n-1}(X)=0,
\]
 then we have
\begin{equation*}
\begin{split}
\label{aq}
&\quad -q(q-1) \int_{\partial \Omega }\psi(\nabla \dot{U}(X)\cdot \nu_{\Omega})|\nabla U|^{q-1}d \mathcal{H}^{n-1}(X)-\int_{\partial \Omega}{\rm tr}(\Pi_{\partial \Omega})\psi^{2}|\nabla U|^{q}d \mathcal{H}^{n-1}(X)\\
&\leq (q-1)\int_{\partial \Omega}(\Pi^{-1}_{\partial \Omega}\nabla_{\partial \Omega} \psi \cdot \nabla_{\partial \Omega} \psi)|\nabla U|^{q}d \mathcal{H}^{n-1}(X),
\end{split}
\end{equation*}
where  $\dot{U}$ satisfies the following equation
\begin{equation*}
\label{Ut}
\left\{
\begin{array}{lr}
\nabla\cdot ((q-2)|\nabla U|^{q-4}(\nabla U \cdot  \nabla \dot{U})\nabla U+|\nabla U|^{q-2}\nabla \dot{U})=0, & X\in {\rnnn} \backslash \bar{\Omega}, \\
\dot{U}(X)=|\nabla U(X)|\psi,   & X\in  \partial \Omega,\\
\lim_{|X|\rightarrow \infty}\dot{U }(X)=0.
\end{array}\right.
\end{equation*}
\end{coro}

\subsection{Torsion}

Our second aim is to give the Poincar\'{e}-type inequality derived by the $L^p$-Brunn-Minkowski inequality for torsion with $p\geq 1$.

Recall the $L_{p}$-Brunn-Minkowski inequality for $T$ \eqref{LPQI}. We obtain the following result along similar lines of proof as in Lemma \ref{LQ7}.
\begin{lem}\label{LQ7d}
\eqref{LPQI} implies that for every one-parameter family $\Pi(h,\phi,I)$ as defined in \eqref{phic}, with $h$, $\phi$ and $\Omega_{t}\in \Pi(h,\phi,I)$,
\begin{equation}\label{lpq7}
\frac{d^{2}}{dt^{2}}[T(\Omega_{t})]\Bigg|_{t=0}T(\Omega_{0})\leq \frac{n+2-p}{n+2}\left( \frac{d}{dt}[T(\Omega_{t})]\Bigg|_{t=0}\right)^{2}.
\end{equation}
\end{lem}

Now, we give the infinitesimal form of the $L_{p}$-Brunn-Minkowski inequality for torsion as follows.
\begin{theo}\label{TP2}
 Let $\Omega \in \koo$ be of class $C^{\infty}_{+}$ and $U$ be the solution of \eqref{Tor2} in $\Omega$.  Let $p\geq 1$. Then for every function $\phi^{p}\in C^{\infty}(\sn)$, we have
\begin{equation}
\begin{split}
\label{aq6w}
&\frac{(p-n-2)}{(n+2)T(\Omega)}\left(\int_{\sn}\phi^{p}h^{1-p}d \mu^{tor}\right)^{2}+ (1-p)\int_{\sn}\phi^{2p}h^{1-2p}d \mu^{tor}\\
&\quad -2 \int_{\sn }\phi^{2p}h^{2-2p}\frac{(\nabla \dot{U}(F(\xi))\cdot \xi)}{\dot{U}(F(\xi))}d \mu^{tor}-\int_{\sn}{\rm tr}((Q^{-1}(h))\phi^{2p}h^{2-2p}d \mu^{tor}\\
&\quad +\frac{2}{p}\int_{\sn}h^{3-3p}\phi^{3p}\frac{1}{\dot{U}(F(\xi))}d \mu^{tor} \leq \int_{\sn}Q^{-1}(h)\nabla_{\sn} \left( \phi^{p}h^{1-p}\right)\cdot \nabla_{\sn} \left(\phi^{p}h^{1-p}\right)d \mu^{tor},
\end{split}
\end{equation}
where $\dot{U}$ satisfies the following equation
\begin{equation*}
\label{Ut}
\left\{
\begin{array}{lr}
\Delta \dot{U}=0, & X\in  \Omega, \\
\dot{U}(X)=\frac{1}{p}\phi^{p}(\nu_{\Omega}(X))h^{1-p}(\nu_{\Omega}(X))|\nabla U(X)|,   & X\in  \partial \Omega.
\end{array}\right.
\end{equation*}
\end{theo}
\begin{proof}

 Assume first that $\phi^{p}\in C^{\infty}(\sn)$. Let $\varepsilon>0$ be such that $(h^{p}+t\phi^{p})^{1/p}\in \mathcal{S}$ for every $t\in (-\varepsilon,\varepsilon)$. Set  $g_{p}(t)=T((h^p+t\phi^p)^{1/p})$ and $Q_{p}(t)=g_{p}(t)^{p/(n+2)}$ for $t\in (-\varepsilon,\varepsilon)$. Recall Lemma \ref{LQ7d}, we know that $Q_{p}$ is concave at $t=0$, then we get
\begin{equation}\label{98a}
Q^{''}_{p}(0)=\frac{p(p-n-2)}{(n+2)^{2}}g_{p}(0)^{\frac{p}{n+2}-2}(g^{'}_{p}(0))^{2}+\frac{p}{n+2}g_{p}(0)^{\frac{p}{n+2}-1}g^{''}_{p}(0)\leq 0.
\end{equation}
Substituting \eqref{cz1} and \eqref{cz2} into \eqref{98a}, thus
\begin{equation}
\begin{split}
\label{ertsw}
&\frac{(p-n-2)}{(n+2)T(\Omega)}\left(\int_{{\sn}}h^{1-p}\phi^{p}|\nabla  U(F(\xi))|^{2}\det(h_{ij}(\xi)+h(\xi)\delta_{ij}){d}{\xi}\right)^{2}\\
&\quad + (1-p)\int_{\sn}h^{1-2p}\phi^{2p}|\nabla U(F(\xi))|^{2}\det(h_{ij}(\xi)+h(\xi)\delta_{ij}){d}{\xi}\\
&\quad -2p\int_{\sn}h^{1-p}\phi^{p}\frac{|\nabla U(F(\xi))|}{\kappa} \nabla \dot{U}(F(\xi))\cdot \xi d \xi\\
&\quad+2\int_{\sn}h^{2-2p}\phi^{2p}|\nabla U(F(\xi))|\det(h_{ij}(\xi)+h(\xi)\delta_{ij}){d}{\xi} \\
&\quad -\int_{\sn}{\rm tr}(c_{ij})|\nabla U(F(\xi))|^{2}h^{2-2p}\phi^{2p}d \xi\\
&\leq -\int_{\sn}\sum_{i,j}h^{1-p}\phi^{p}(c_{ij}|\nabla U(F(\xi))|^{2}(h^{1-p}\phi^p)_{i})_{j}d \xi.
\end{split}
\end{equation}
Now, with an integration by parts, we have
\begin{equation}
\begin{split}
\label{a16w}
&\frac{(p-n-2)}{(n+2)T(\Omega)}\left(\int_{{\sn}}h^{1-p}\phi^{p}|\nabla  U(F(\xi))|^{2}\det(h_{ij}(\xi)+h(\xi)\delta_{ij}){d}{\xi}\right)^{2}\\
&\quad + (1-p)\int_{\sn}h^{1-2p}\phi^{2p}|\nabla U(F(\xi))|^{2}\det(h_{ij}(\xi)+h(\xi)\delta_{ij}){d}{\xi}\\
&\quad -2p\int_{\sn}h^{1-p}\phi^{p}\frac{|\nabla U(F(\xi))|}{\kappa} \nabla \dot{U}(F(\xi))\cdot \xi d \xi\\
&\quad+2\int_{\sn}h^{2-2p}\phi^{2p}|\nabla U(F(\xi))|\det(h_{ij}(\xi)+h(\xi)\delta_{ij}){d}{\xi} \\
&\quad -\int_{\sn}{\rm tr}(c_{ij})|\nabla U(F(\xi))|^{2}h^{2-2p}\phi^{2p}d \xi\\
&\leq \int_{\sn}\sum_{i,j}(h^{1-p}\phi^{p})_{j}(c_{ij}|\nabla U(F(\xi))|^{2}(h^{1-p}\phi^p)_{i})d \xi.
\end{split}
\end{equation}
Note that
\begin{equation}\label{cijh}
c_{ij}=\det(h_{ij}+h\delta_{ij})(h_{ij}+h\delta_{ij})^{-1}.
\end{equation}
Then \eqref{aq6w} holds by substituting \eqref{cijh} and the definition of $\mu^{tor}$ into \eqref{a16w}. The proof is completed.

\end{proof}
 We obtain the following result via Theorem \ref{TP2} and along similar lines of proof in proving Theorem \ref{main1}.
\begin{theo}\label{main2}

Let $\Omega \in \koo$  be of class $C^{\infty}_{+}$ and $U$ be the solution of \eqref{Tor2} in $\Omega$. Suppose $1\leq p < \infty$. For every function $\psi^{p}\in C^{\infty}(\partial \Omega)$, then we obtain
\begin{equation}
\begin{split}
\label{aq9}
&\frac{(p-n-2)}{(n+2)T(\Omega)}\left(\int_{\partial \Omega}\frac{\psi^{p}}{|\nabla h^{*}|^{1-p}}|\nabla U|^{2}d \mathcal{H}^{n-1}(X)\right)^{2}\\
&\quad + (1-p)\int_{\partial \Omega}\frac{\psi^{2p}}{|\nabla h^{*}|^{1-2p}}|\nabla U|^{2}d \mathcal{H}^{n-1}(X)\\
&\quad -2p \int_{\partial \Omega }\frac{\psi^{p}}{|\nabla h^{*}|^{1-p}}(\nabla \dot{U}(X)\cdot \nu_{\Omega})|\nabla U|d \mathcal{H}^{n-1}(X)\\
&\quad+2\int_{\partial \Omega} \frac{\psi^{2p}}{|\nabla h^{*}|^{2-2p}}|\nabla U|d \mathcal{H}^{n-1}(X) -\int_{\partial \Omega}{\rm tr}(\Pi_{\partial\Omega})\frac{\psi^{2p}}{|\nabla h^{*}|^{2-2p}}|\nabla U|^{2}d \mathcal{H}^{n-1}(X)\\
&\leq \int_{\partial \Omega}(\Pi^{-1}_{\partial \Omega})\nabla_{\partial \Omega} \left( \frac{\psi^{p}}{|\nabla h^{*}|^{1-p}}\right)\cdot \nabla_{\partial \Omega} \left( \frac{\psi^{p}}{|\nabla h^{*}|^{1-p}}\right)|\nabla U|^{2}d \mathcal{H}^{n-1}(X),
\end{split}
\end{equation}
where $\dot{U}$ satisfies the following equation
\begin{equation*}
\label{Ut}
\left\{
\begin{array}{lr}
\Delta \dot{U}=0, & X\in \Omega, \\
\dot{U}(X)=\frac{\psi^{p}}{p|\nabla h^{*}|^{1-p}}|\nabla U(X)|,   & X\in  \partial \Omega.
\end{array}\right.
\end{equation*}
\end{theo}
\begin{rem}
Notice that if $\Omega=B^{n}$ in \eqref{aq9},  when $\psi(\xi)^{p}=\xi \cdot v_{0}$ for a fixed vector $v_{0}$ in $\rnnn$ and $\xi\in B^{n}$, in such case, by the maximum principle, it follows that $\dot{U}=\frac{\psi^p}{pn}$ in $B^{n}$,  then \eqref{aq9} always holds with $1\leq p<\infty$ (when $p=1$, equality holds). Remark also that a class of Poincar\'{e}-type inequalities for torsion derived by the related Brunn-Minkowski inequality were previously studied in \cite{FH24}. In addition, we can use Theorem \ref{main2} to give a new proof on \eqref{LPQI} in small enough $C^{\infty}$-perturbations of unit ball (see Sec. \ref{Sec5}).
\end{rem}

\subsection{The first eigenvalue}

We establish a Poincar\'{e} type inequality coming from the Brunn-Minkowski inequality for the first eigenvalue.

Recall the Brunn-Minkowski inequality for $\lambda$ \eqref{Cta}. With a proof similar to that of Lemma \ref{LQ7}, the following result holds.
\begin{lem}\label{LQsg}
\eqref{Cta} implies that for every one-parameter family $\Pi(h,\phi,I)$ as given in \eqref{phic}, with $h$, $\phi$, $\Omega_{t}\in \Pi(h,\phi,I)$ and $p=1$,
\begin{equation}\label{lp6}
\frac{d^{2}}{dt^{2}}[\lambda(\Omega_{t})]\Bigg|_{t=0}\lambda(\Omega_{0})\geq  \frac{3}{2}\left( \frac{d}{dt}[\lambda(\Omega_{t})]\Bigg|_{t=0}\right)^{2}.
\end{equation}

\end{lem}

The related Poincar\'{e} type inequality's form is as follows.
\begin{theo}\label{TP36}
 Let $\Omega \in \koo$ be of class $C^{\infty}_{+}$ and $U$ be the solution of \eqref{MA} in $\Omega$. Then for every function $\phi\in C^{\infty}(\sn)$, we have
\begin{equation}
\begin{split}
\label{aq66}
&\frac{3}{\lambda(\Omega)}\left(\int_{\sn}\phi d \mu^{eig}\right)^{2}-2 \int_{\sn }\phi^{2}\frac{(\nabla \dot{U}(F(\xi))\cdot \xi)}{\dot{U}(F(\xi))}d \mu^{eig}-\int_{\sn}{\rm tr}((Q^{-1}(h))\phi^{2}d \mu^{eig}\\
&\quad +2\int_{\sn}\lambda U\phi^{3}\frac{1}{\dot{U}(F(\xi))}d \mu^{eig} \leq \int_{\sn}Q^{-1}(h)\nabla_{\sn}  \phi\cdot \nabla_{\sn} \phi d \mu^{eig},
\end{split}
\end{equation}
where $\dot{U}$ satisfies the following equation
\begin{equation*}
\label{Ut}
\left\{
\begin{array}{lr}
\Delta \dot{U}=-\lambda \dot{U}, & X\in  \Omega, \\
\dot{U}(X)=\phi(\nu_{\Omega}(X))|\nabla U(X)|,   & X\in  \partial \Omega.
\end{array}\right.
\end{equation*}
\end{theo}
\begin{proof}

 Assume first that $\phi\in C^{\infty}(\sn)$. Let $\varepsilon>0$ be such that $h+t\phi\in \mathcal{S}$ for every $t\in (-\varepsilon,\varepsilon)$. Set  $\Lambda(t)=\lambda(h+t\phi)$ and $W(t)=\Lambda(t)^{-1/2}$ for $t\in (-\varepsilon,\varepsilon)$. Recall Lemma \ref{LQsg}, we know that $W$ is concave at $t=0$, then we get
\begin{equation}\label{1w}
W^{''}(0)=\frac{3}{4}\Lambda(0)^{-\frac{1}{2}-2}(\Lambda^{'}(0))^{2}-\frac{1}{2}\Lambda(0)^{-\frac{1}{2}-1}\Lambda^{''}(0)\leq 0.
\end{equation}
Substituting \eqref{1p} and \eqref{2p}  into \eqref{1w}, we derive
\begin{equation}
\begin{split}
\label{ert6}
&\frac{3}{2\lambda(\Omega)}\left(\int_{{\sn}}\phi|\nabla  U(F(\xi))|^{2}\det(h_{ij}(\xi)+h(\xi)\delta_{ij}){d}{\xi}\right)^{2}\\
&\quad -2\int_{\sn}\phi\frac{|\nabla U(F(\xi))|}{\kappa} \nabla \dot{U}(F(\xi))\cdot \xi d \xi\\
&\quad+2\int_{\sn}\lambda U\phi^{2}|\nabla U(F(\xi))|\det(h_{ij}(\xi)+h(\xi)\delta_{ij}){d}{\xi} \\
&\quad -\int_{\sn}{\rm tr}(c_{ij})|\nabla U(F(\xi))|^{2}\phi^{2}d \xi\\
&\leq -\int_{\sn}\sum_{i,j}\phi(c_{ij}|\nabla U(F(\xi))|^{2}\phi_{i})_{j}d \xi.
\end{split}
\end{equation}
Employing an integration by parts, we have
\begin{equation}
\begin{split}
\label{a166}
&\frac{3}{2\lambda(\Omega)}\left(\int_{{\sn}}\phi|\nabla  U(F(\xi))|^{2}\det(h_{ij}(\xi)+h(\xi)\delta_{ij}){d}{\xi}\right)^{2}\\
&\quad -2\int_{\sn}\phi\frac{|\nabla U(F(\xi))|}{\kappa} \nabla \dot{U}(F(\xi))\cdot \xi d \xi\\
&\quad+2\int_{\sn}\lambda U\phi^{2}|\nabla U(F(\xi))|\det(h_{ij}(\xi)+h(\xi)\delta_{ij}){d}{\xi} \\
&\quad -\int_{\sn}{\rm tr}(c_{ij})|\nabla U(F(\xi))|^{2}\phi^{2}d \xi\\
&\leq \int_{\sn}\sum_{i,j}\phi_{j}c_{ij}|\nabla U(F(\xi))|^{2}\phi_{i}d \xi.
\end{split}
\end{equation}
Note that
\begin{equation}\label{cijh6}
c_{ij}=\det(h_{ij}+h\delta_{ij})(h_{ij}+h\delta_{ij})^{-1}.
\end{equation}
Then \eqref{aq66} holds by substituting \eqref{cijh6} and the definition of $\mu^{eig}$ into \eqref{a166}. The proof is completed.

\end{proof}

With the aid of Theorem \ref{TP36}, following again the proof of Theorem \ref{main1}, we can also get the following result.

\begin{theo}\label{main3}
Let $\Omega \in \koo$ be of class $C^{\infty}_{+}$ and $U$ be the solution of \eqref{MA} in $\Omega$.  For every function $\psi\in C^{\infty}(\partial \Omega)$, then we obtain
\begin{equation}
\begin{split}
\label{aq96}
&\frac{3}{2\lambda(\Omega)}\left(\int_{\partial \Omega}\psi|\nabla U|^{2}d \mathcal{H}^{n-1}(X)\right)^{2}-2\int_{\partial \Omega }\psi(\nabla \dot{U}(X)\cdot \nu_{\Omega})|\nabla U|d \mathcal{H}^{n-1}(X)\\
&\quad-\int_{\partial \Omega}{\rm tr}(\Pi_{\partial\Omega})\psi^{2}|\nabla U|^{2}d \mathcal{H}^{n-1}(X)\\
&\leq \int_{\partial \Omega}(\Pi^{-1}_{\partial \Omega}) \nabla_{\partial \Omega} \psi\cdot \nabla_{\partial \Omega} \psi|\nabla U|^{2}d \mathcal{H}^{n-1}(X).
\end{split}
\end{equation}
In particular, if
\[
\int_{\partial \Omega}\psi|\nabla U|^{2}d \mathcal{H}^{n-1}(X)=0,
\]
then
\begin{equation*}
\begin{split}
\label{Up3}
&-2\int_{\partial \Omega }\psi(\nabla \dot{U}(X)\cdot \nu_{\Omega})|\nabla U|d \mathcal{H}^{n-1}(X)-\int_{\partial \Omega}{\rm tr}(\Pi_{\partial\Omega})\psi^{2}|\nabla U|^{2}d \mathcal{H}^{n-1}(X)\\
&\leq \int_{\partial \Omega}(\Pi^{-1}_{\partial \Omega}) \nabla_{\partial \Omega} \psi\cdot \nabla_{\partial \Omega} \psi|\nabla U|^{2}d \mathcal{H}^{n-1}(X),
\end{split}
\end{equation*}
where $\dot{U}$ satisfies the following equation
\begin{equation*}
\label{Ut}
\left\{
\begin{array}{lr}
\Delta \dot{U}=-\lambda \dot{U}, & X\in \Omega, \\
\dot{U}(X)=|\nabla U(X)|\psi,   & X\in  \partial \Omega.
\end{array}\right.
\end{equation*}
\end{theo}
\begin{rem}
Above Poincar\'{e}-type inequalities constitute the infinitesimal versions of the (sharp) $L_{p}$-Brunn-Minkowski inequalities for variational functionals with $p\geq 1$, and these inequalities are sharp since related Brunn-Minkowski inequalities are sharp.

\end{rem}
\section{Local $L_{p}$-Brunn-Minkowski inequalities for torsion with $0\leq p<1$}
\label{Sec5}

Before proving Theorems \ref{main5} and \ref{main6}, we first introduce the following problem about the infinitesimal formulations of the $L_{p}$-Brunn-Minkowski inequalities for torsion.

{\bf Problem A:} Let $p\in [0,1)$. For all $\Omega \in \mathcal{K}^{\infty}_{+,e}$:
\begin{equation}\label{PI}
\forall \varphi^{p}\in C^{\infty}_{e}(\sn)  \quad \frac{d^{2}}{(dt)^{2}}\Big|_{t=0}T(\Omega+_{p}t\cdot \varphi)^{\frac{p}{n+2}}\leq 0.
\end{equation}
When $p=0$, \eqref{PI} is interpreted as
\begin{equation}\label{PII}
\forall \ positive \ \varphi\in C^{\infty}_{e}(\sn)\quad \frac{d^{2}}{(d t)^{2}}\Big|_{t=0}\log T(\Omega+_{0}t\cdot \varphi)\leq 0.
\end{equation}

\begin{prop}\label{ppo1}
Let $\Omega\in \mathcal{K}^{\infty}_{+,e}$ and $U$ be the solution of \eqref{Tor2} in $\Omega$. Given $0<p<1$, \eqref{PI} for $\Omega$ is equivalent to the assertion that for  $\phi^{p}\in C^{\infty}_{e}(\sn)$,
\begin{equation}
\begin{split}
\label{aq6wl}
&\frac{(p-n-2)}{(n+2)T(\Omega)}\left(\int_{\sn}\phi^{p}h^{1-p}d \mu^{tor}\right)^{2}+ (1-p)\int_{\sn}\phi^{2p}h^{1-2p}d \mu^{tor}\\
&\quad -2 \int_{\sn }\phi^{2p}h^{2-2p}\frac{(\nabla \dot{U}(F(\xi))\cdot \xi)}{\dot{U}(F(\xi))}d \mu^{tor}-\int_{\sn}{\rm tr}((Q^{-1}(h))\phi^{2p}h^{2-2p}d \mu^{tor}\\
&\quad +\frac{2}{p}\int_{\sn}h^{3-3p}\phi^{3p}\frac{1}{\dot{U}(F(\xi))}d \mu^{tor} \leq \int_{\sn}Q^{-1}(h)\nabla_{\sn} \left( \phi^{p}h^{1-p}\right)\cdot \nabla_{\sn} \left(\phi^{p}h^{1-p}\right)d \mu^{tor},
\end{split}
\end{equation}
where $\dot{U}$ satisfies the following equation
\begin{equation*}
\label{Ut}
\left\{
\begin{array}{lr}
\Delta \dot{U}=0, & X\in  \Omega, \\
\dot{U}(X)=\frac{1}{p}\phi^{p}(\nu_{\Omega}(X))h^{1-p}(\nu_{\Omega}(X))|\nabla U(X)|,   & X\in  \partial \Omega.
\end{array}\right.
\end{equation*}
\end{prop}
\begin{proof}
The proof is similar to the proof of Theorem \ref{TP2}.
\end{proof}

Proposition \ref{ppo1} leads to the following outcome.

\begin{prop}\label{ppo2}
Let $\Omega\in \mathcal{K}^{\infty}_{+,e}$ and $U$ be the solution of \eqref{Tor2} in $\Omega$. Given $0<p<1$, \eqref{PI} for $\Omega$ is equivalent to the assertion that for  $\psi^{p}\in C^{\infty}_{e}(\partial \Omega)$,
\begin{equation}
\begin{split}
\label{aq9q}
&\frac{(p-n-2)}{(n+2)T(\Omega)}\left(\int_{\partial \Omega}\frac{\psi^{p}}{|\nabla h^{*}|^{1-p}}|\nabla U|^{2}d \mathcal{H}^{n-1}(X)\right)^{2}\\
&\quad + (1-p)\int_{\partial \Omega}\frac{\psi^{2p}}{|\nabla h^{*}|^{1-2p}}|\nabla U|^{2}d \mathcal{H}^{n-1}(X)\\
&\quad -2p \int_{\partial \Omega }\frac{\psi^{p}}{|\nabla h^{*}|^{1-p}}(\nabla \dot{U}(X)\cdot \nu_{\Omega})|\nabla U|d \mathcal{H}^{n-1}(X)\\
&\quad+2\int_{\partial \Omega} \frac{\psi^{2p}}{|\nabla h^{*}|^{2-2p}}|\nabla U|d \mathcal{H}^{n-1}(X) -\int_{\partial \Omega}{\rm tr}(\Pi_{\partial\Omega})\frac{\psi^{2p}}{|\nabla h^{*}|^{2-2p}}|\nabla U|^{2}d \mathcal{H}^{n-1}(X)\\
&\leq \int_{\partial \Omega}(\Pi^{-1}_{\partial \Omega})\nabla_{\partial \Omega} \left( \frac{\psi^{p}}{|\nabla h^{*}|^{1-p}}\right)\cdot \nabla_{\partial \Omega} \left( \frac{\psi^{p}}{|\nabla h^{*}|^{1-p}}\right)|\nabla U|^{2}d \mathcal{H}^{n-1}(X),
\end{split}
\end{equation}
where $\dot{U}$ satisfies the following equation
\begin{equation*}
\label{Ut}
\left\{
\begin{array}{lr}
\Delta \dot{U}=0, & X\in \Omega, \\
\dot{U}(X)=\frac{\psi^{p}}{p|\nabla h^{*}|^{1-p}}|\nabla U(X)|,   & X\in  \partial \Omega.
\end{array}\right.
\end{equation*}
\end{prop}

The proof of the following lemma is similar to the one for Lemma \ref{fsd5}.

\begin{lem}\label{fsd7}
 Let $\Omega_{t}\in \mathcal{K}^{\infty}_{+,e}$ with support function $h_{t}$, and $U(\cdot,t)$ be the solution of \eqref{Tor2} in $\Omega_{t}$. Let  $\tilde{\phi}\in C^{\infty}_{e}(\sn)$ be positive,  and for  $\varepsilon>0$ such that $h\tilde{\phi}^{t}\in \mathcal{S}$ for every $t\in (-\varepsilon,\varepsilon)$. Set $h_{t}=h\tilde{\phi}^{t}$ and $g(t)=T(h_{t})$. Then
\begin{equation}\label{tyr}
g^{'}(0)=\int_{\sn}h \log \tilde{\phi}\frac{ |\nabla U(F(\xi))|^{2}}{\kappa}d \xi,
\end{equation}
and
\begin{equation}
\begin{split}
\label{tyr2}
g^{''}(0)&=\int_{\sn}h(\log \tilde{\phi})^{2}|\nabla U(F(\xi))|^{2}\frac{1}{\kappa}d\xi\\
&\quad -2\int_{\sn}h\log \tilde{\phi} \frac{|\nabla U(F(\xi))|}{\kappa}\nabla \dot{U}(F(\xi))\cdot \xi d \xi\\
&\quad-\int_{\sn}\sum_{i}|\nabla U(F(\xi))|^{2}c_{ii}(h\log \tilde{\phi})^{2}d \xi\\
&\quad +2\int_{\sn}(h\log \tilde{\phi})^{2}\frac{|\nabla U(F(\xi))|}{\kappa}d \xi\\
&\quad +\int_{\sn}\sum_{i,j}h\log \tilde{\phi}(c_{ij}|\nabla U(F(\xi))|^{2}(h\log \tilde{\phi})_{i})_{j}d \xi,
\end{split}
\end{equation}
where $\dot{U}$ satisfies the following equation
\begin{equation*}
\label{Ut}
\left\{
\begin{array}{lr}
\Delta \dot{U}=0, & X\in \Omega, \\
\dot{U}(X)=|\nabla U(X)|h\log\tilde{\phi}(\nu_{\Omega}(X)),   & X\in  \partial \Omega.
\end{array}\right.
\end{equation*}
\end{lem}

The infinitesimal formulation of \eqref{PII} is as follows.
\begin{prop}\label{TP9o}
Let $\Omega \in \mathcal{K}^{\infty}_{+,e}$  and $U$ be the solution of \eqref{Tor2} in $\Omega$. \eqref{PII} for $\Omega$ is equivalent to the assertion that for $\phi\in C^{\infty}_{e}(\sn)$,
\begin{equation}
\begin{split}
\label{LP228}
&\int_{ \sn}\frac{\phi^{2}}{h^{2}}dG^{tor} -\frac{(n+2)}{T(\Omega)}\left( \int_{\sn}\frac{\phi}{h} dG^{tor} \right)^{2}+2\int_{\sn}\frac{\phi^{3}}{h}\frac{1}{\dot{U}(F(\xi))}d G^{tor}\\
&-2\int_{\sn} \frac{\phi^{2}}{h}\frac{(\nabla \dot{U}(F(\xi))\cdot \xi)}{\dot{U}(F(\xi))}d G^{tor}-\int_{\sn}{\rm tr}(Q^{-1}(h))\frac{\phi^{2}}{h}d G^{tor}\\
&\leq \int_{\sn}\frac{1}{h}Q^{-1}(h)\nabla_{\sn} \phi \cdot \nabla_{\sn} \phi d G^{tor},
\end{split}
\end{equation}
where $\dot{U}$ satisfies the following equation
\begin{equation*}
\label{Ut}
\left\{
\begin{array}{lr}
\Delta \dot{U}=0, & X\in \Omega, \\
\dot{U}(X)=|\nabla U(X)|\phi(\nu_{\Omega}(X)),   & X\in  \partial \Omega.
\end{array}\right.
\end{equation*}
\end{prop}
\begin{proof}
 Assume that  $\tilde{\phi}\in C^{\infty}_{e}(\sn)$ with $\tilde{\phi}>0$. Let $\varepsilon>0$ be such that $h\tilde{\phi}^{t}\in \mathcal{S}$ for every $t\in (-\varepsilon,\varepsilon)$. Set  $Q(t)=T(h\tilde{\phi}^{t})$, so $\log Q(t)$ is concave at $t=0$, then we get
\begin{equation}\label{kkl}
Q^{''}(0)-\frac{(Q^{'}(0))^{2}}{Q(0)}\leq 0.
\end{equation}
Applying \eqref{tyr} and \eqref{tyr2} into \eqref{kkl}, there is
\begin{equation}
\begin{split}
\label{Lm67}
&\int_{\sn}h(\log \tilde{\phi})^{2}|\nabla U(F(\xi))|^{2}\frac{1}{\kappa}d\xi\\
&\quad -2\int_{\sn}h\log \tilde{\phi} \frac{|\nabla U(F(\xi))|}{\kappa}\nabla \dot{U}(F(\xi))\cdot \xi d \xi\\
&\quad +2\int_{\sn}(h\log \tilde{\phi})^{2}\frac{|\nabla U(F(\xi))|}{\kappa}d \xi\\
&\quad-\int_{\sn}\sum_{i}|\nabla U(F(\xi))|^{2}c_{ii}(h\log \tilde{\phi})^{2}d \xi-\frac{(\int_{\sn}h \log \tilde{\phi}\frac{ |\nabla U(F(\xi))|^{2}}{\kappa}d \xi)^{2}}{T(h)}\\
&\quad \leq \int_{\sn}\sum_{i,j}(h\log \tilde{\phi})_{j}(c_{ij}|\nabla U(F(\xi))|^{2}(h\log \tilde{\phi})_{i})d \xi.
\end{split}
\end{equation}
Now let $\phi \in C^{\infty}_{e}( \sn)$ and set $h(\xi)\log\tilde{\phi}(\xi)=\phi(\xi)$ for every $\xi\in \sn$. Using $dG^{tor}=\frac{1}{n+2}h d\mu^{tor}$ and along the lines in proving Theorem \ref{TP2}, we can transfer \eqref{Lm67} into
\begin{equation*}
\begin{split}
\label{Up3}
&(n+2)\int_{ \sn}\frac{\phi^{2}}{h^{2}}dG^{tor} -\frac{(n+2)^{2}}{T(\Omega)}\left( \int_{\sn}\frac{\phi}{h} dG^{tor} \right)^{2}+2(n+2)\int_{\sn}\frac{\phi^{3}}{h}\frac{1}{\dot{U}(F(\xi))}d G^{tor}\\
&-2(n+2)\int_{\sn} \frac{\phi^{2}}{h}\frac{(\nabla \dot{U}(F(\xi))\cdot \xi)}{\dot{U}(F(\xi))}d G^{tor}-(n+2)\int_{\sn}{\rm tr}(Q^{-1}(h))\frac{\phi^{2}}{h}d G^{tor}\\
&\leq (n+2)\int_{\sn}\frac{1}{h}Q^{-1}(h)\nabla_{\sn} \phi \cdot \nabla_{\sn} \phi d G^{tor}.
\end{split}
\end{equation*}
 The proof is completed.
\end{proof}

A natural consequence of Proposition \ref{TP9o} is as follows.

\begin{prop}\label{TP}
Let $\Omega \in \mathcal{K}^{\infty}_{+,e}$  and $U$ be the solution of \eqref{Tor2} in $\Omega$. \eqref{PII} for $\Omega$ is equivalent to the assertion that for $\psi\in C^{\infty}_{e}(\partial \Omega)$,
\begin{equation}
\begin{split}
\label{LP2}
&\int_{\partial \Omega}\frac{\psi^{2}}{|\nabla h^{*}|^{-1}}|\nabla U|^{2}d \mathcal{H}^{n-1} -\frac{1}{T(\Omega)}\left( \int_{\partial \Omega}\psi|\nabla U|^{2}d \mathcal{H}^{n-1} \right)^{2}+2\int_{\partial \Omega}\psi^{2}|\nabla U|d\mathcal{H}^{n-1}\\
&-2\int_{\partial \Omega} \psi^{2}\frac{(\nabla \dot{U}(X)\cdot \nu_{\Omega})}{\dot{U}(X)}|\nabla U|^{2}d \mathcal{H}^{n-1} -\int_{\partial \Omega}{\rm tr}(\Pi_{\partial\Omega})\psi^{2}|\nabla U|^{2}d \mathcal{H}^{n-1}\\
&\leq  \int_{\partial \Omega}((\Pi^{-1}_{\partial\Omega})\nabla_{\partial \Omega} \psi \cdot \nabla_{\partial \Omega} \psi) |\nabla U|^{2}d \mathcal{H}^{n-1},
\end{split}
\end{equation}
where $h^{*}$ is the support function of the polar body of $\Omega$, and $\dot{U}$ satisfies the following equation
\begin{equation*}
\label{Ut}
\left\{
\begin{array}{lr}
\Delta \dot{U}=0, & X\in \Omega, \\
\dot{U}(X)=|\nabla U(X)|\psi,   & X\in  \partial \Omega.
\end{array}\right.
\end{equation*}
\end{prop}
\begin{rem}

In \eqref{aq6wl}, set $\hat{\phi}:=h^{1-p}\phi^{p}$,  and apply $dG^{tor}=\frac{1}{n+2}h d\mu^{tor}$. Then for $\hat{\phi} \in C^{\infty}_{e}(\sn)$, we get the corresponding infinitesimal  $L_{p}$-Brunn-Minkowski inequality for torsion as follows:
\begin{equation*}
\begin{split}
\label{aq68}
&\left(\frac{p}{T(\Omega)}-\frac{(n+2)}{T(\Omega)}\right)\left(\int_{\sn}\frac{\hat{\phi}}{h} d G^{tor}\right)^{2} +(1-p)\int_{\sn}\frac{\hat{\phi}^{2}}{h^{2}}dG^{tor}\\
&-2\int_{\sn }\frac{\hat{\phi}^{2}}{h}\frac{(\nabla \dot{U}(F(\xi))\cdot \xi)}{\dot{U}(F (\xi))}d G^{tor} +2\int_{\sn}\frac{\hat{\phi}^{3}}{h}\frac{1}{\dot{U}(F(\xi))}dG^{tor}-\int_{\sn}{\rm tr}(Q^{-1}(h))\frac{\hat{\phi}^{2}}{h}d G^{tor}\\
&\leq \int_{\sn}\frac{1}{h}(Q^{-1}(h)\nabla_{\sn} \hat{\phi} \cdot \nabla_{\sn} \hat{\phi}) d G^{tor},
\end{split}
\end{equation*}
where $\dot{U}$ satisfies the following equation
\begin{equation*}
\label{Ut}
\left\{
\begin{array}{lr}
\Delta \dot{U}=0, & X\in \Omega, \\
\dot{U}(X)=|\nabla U(X)|\hat{\phi}(\nu_{\Omega}(X)),   & X\in  \partial \Omega.
\end{array}\right.
\end{equation*}
 By the Cauchy-Schwarz inequality,  one sees
\begin{equation*}
\begin{split}
\label{Up3}
&p\int_{\sn}\frac{\hat{\phi}^{2}}{h^{2}}dG^{tor}-\frac{p}{T(\Omega)}\left( \int_{\sn}\frac{\hat{\phi}}{h} dG^{tor} \right)^{2}\\
&=pT(\Omega)\left [\int_{\sn}\frac{\hat{\phi}^{2}}{h^{2}}d\bar{G}^{tor}-\left( \int_{\sn}\frac{\hat{\phi}}{h} d\bar{G}^{tor} \right)^{2} \right]\geq 0.
\end{split}
\end{equation*}
This implies that \eqref{LP228} is indeed a strengthening of \eqref{aq6wl}.
\end{rem}

The following lemmas are essential in proving Theorems \ref{main5} and \ref{main6}.

\begin{lem}\label{LQ7p}
Suppose $0<p< 1$. If for every one-parameter family $\Pi(1,\phi,I)$ as defined in \eqref{phic} with $h=1$ and even $\phi$ such that $\Omega_{t}\in \Pi(1,\phi,I)$,
\begin{equation}\label{lpq76o}
\frac{d^{2}}{dt^{2}}[T(\Omega_{t})]\Bigg|_{t=0}T(B^{n})< \frac{n+2-p}{n+2}\left( \frac{d}{dt}[T(\Omega_{t})]\Bigg|_{t=0}\right)^{2},
\end{equation}
 then for any $\Omega_{1}, \Omega_{2}\in \Pi(1,\phi,I_{0})$ where $I_{0}=[-\varepsilon_{0}, \varepsilon_{0}]\subseteq I$ with $\varepsilon_{0}>0$, inequality \eqref{LPN} holds for $\Omega_{1}, \Omega_{2}$.
\end{lem}
\begin{proof}
Suppose that for $\Pi(1,\phi,I)$ as defined in \eqref{phic} with $h=1$ and even $\phi$, the function $T(\Omega_{t})^{\frac{p}{n+2}} $ has negative second derivative at $t=0$, i.e., \eqref{lpq76o} holds. Now, given $t_{0}$ in the interior of $I$, consider $\hat{1}=( 1+t_{0}\phi^{p})^{1/p}$, and set a new parameter system $\Pi(\hat{1},\phi,\hat{I})$, where $\hat{I}$ is a new interval such that $(\hat{1}^{p}+t\phi^{p})^{1/p}=(1+(t+t_0)\phi^{p})^{1/p}\in \mathcal{S}$ for every $t\in \hat{I}$. By \eqref{lpq76o}, one sees that the second derivative of $ T(\Omega_{t})^{\frac{p}{n+2}}$ at $t=t_{0}$ is non-positive, which is equivalent to the non-positivity of the second derivative of  $T^{\frac{p}{n+2}}(\hat{\Omega}_{t})$ at $t=0$ with $\hat{\Omega }_{t}\in \Pi(\hat{1},\phi,\hat{I})$, i.e.,
\begin{equation}\label{TUW}
 \frac{d^{2}}{dt^{2}}\Big|_{t=0}T(\Omega+_{p}t\cdot \phi)^{\frac{p}{n+2}} \leq 0
 \end{equation}
 holds for all $\Omega \in \Pi(1,\phi, I_{0})$ with $I_{0}\subseteq I$. Then for all $\Omega_{1}, \Omega_{2}\in \Pi(1,\phi,I_{0})$, the global concavity of $[0,1]\ni s\mapsto T((1-s)\cdot \Omega_{1}+_{p}s\cdot \Omega_{2})^{\frac{p}{n+2}}$ follows by testing its second derivative at a given $s\in [0,1]$, which is non-positive by \eqref{TUW} for the body $\Omega= (1-s)\cdot \Omega_{1}+_{p}s\cdot \Omega_{2}$ and the function $\phi^{p}:=(h^{p}_{\Omega_{2}}-h^{p}_{\Omega_{1}})\in C^{\infty}_{e}(\sn)$. So, we assert that this lemma holds.
\end{proof}

Let $h\in \mathcal{S}$ with $h>0$ be the support function of a convex body $\Omega\in \koo$ of class $C^{\infty}_{+}$, and let $\tilde{\phi}\in C^{\infty}(\sn)$ with $\tilde{\phi}>0$, then there exists a sufficiently small $\varepsilon>0$ such that
\[
h_{t}:=h\tilde{\phi}^{t}\in \mathcal{S}, \ \forall t\in [-\varepsilon,\varepsilon].
\]
For an interval $I:=[-\varepsilon,\varepsilon]$, we set the one-parameter family of convex bodies $\Omega_{t}\in \koo$ of class $C^{\infty}_{+}$:
\begin{equation}\label{Pid}
\Pi(h,\tilde{\phi},I)=\{ \Omega_{t}: h_{\Omega_{t}}= h_{t}= h\tilde{\phi}^{t}, t\in I \}.
\end{equation}

Similar to Lemma \ref{LQ7p}, we have the following lemma.
\begin{lem}\label{LE}

 If for every one-parameter family $\Pi(1,\tilde{\phi},I)$ as defined in \eqref{Pid} with $h=1$ and even $\tilde{\phi}$ such that $\Omega_{t}\in \Pi(1,\tilde{\phi},I)$,
\begin{equation}\label{lpq76g}
\frac{d^{2}}{dt^{2}}[\log T(\Omega_{t})]\Bigg|_{t=0}< 0,
\end{equation}
 then for any $\Omega_{1}, \Omega_{2}\in\Pi(1,\tilde{\phi},I_{0})$ where $I_{0}=[-\varepsilon_{0},\varepsilon_{0}]\subseteq I$ with $\varepsilon_{0}>0$, inequality \eqref{LPN2} holds for $\Omega_{1}, \Omega_{2}$.
\end{lem}

{\bf Proof of Theorem \ref{main5}}

In the case $0<p<1$: on the one hand, if $\varphi^{p}=constant$, the $L_{p}$-Brunn-Minkowski inequalities for torsion \eqref{LPN}  hold in the case of Euclidean balls. On the other hand, if $\varphi^{p}\neq constant$, by Lemma \ref{LQ7p} and Proposition \ref{ppo2}, we conclude that to prove \eqref{LPN}, it suffices to verify the (strict) validity of \eqref{aq9q} in the case that $\Omega$ is the unit ball $B^{n}$. Then, $\partial \Omega=\sn$, $T(B^{n})=\frac{\omega_{n}}{n(n+2)}$, $|\nabla h^{*}(X)|=1$, $\nu_{\Omega}$ is identity map, and $|\nabla U(X)|=\frac{1}{n}$ for $X\in \partial \Omega$. In such a case, we verify that the following theorem holds.

\begin{theo}\label{TLP}
Suppose $n\geq 2$ and $p>0$. Let $\varphi^{p} \in C^{\infty}(\sn)$ be even with $\varphi^p\neq constant$. Then, the following inequality
\begin{equation}
\begin{split}
\label{bH0}
& \frac{(p-2-n)}{n\omega_{n}}\left( \int_{\sn} \varphi^{p}  d \mathcal{H}^{n-1} \right)^{2}+(n+2-p)\int_{\sn}\varphi^{2p}d \mathcal{H}^{n-1}\\
&-2pn\int_{\sn}\varphi^{p}(\nabla \dot{U}(\xi)\cdot\xi) d \mathcal{H}^{n-1}< \int_{\sn}\nabla_{\sn} \varphi^{p}\cdot \nabla_{\sn} \varphi^{p}d \mathcal{H}^{n-1}
\end{split}
\end{equation}
holds, where
$\dot{U}$ satisfies the following equation
\begin{equation}
\label{Ut0}
\left\{
\begin{array}{lr}
\Delta\dot{U}=0, &\xi\in B^{n}, \\
\dot{U}(\xi)=\frac{\varphi^p}{pn},   & \xi\in \sn.
\end{array}\right.
\end{equation}
\end{theo}

\begin{proof}
By using \eqref{Ut0} and the divergence theorem, there holds
\begin{equation}
\begin{split}
\label{iu0y}
&2pn\int_{\sn}\varphi^{p}(\nabla \dot{U}(\xi)\cdot \xi)d \mathcal{H}^{n-1}\\
&=2(pn)^{2}\int_{\sn}\dot{U}(\xi)\nabla \dot{U}(\xi)\cdot \xi d \mathcal{H}^{n-1}\\
&=2(pn)^{2}\int_{B^{n}}{\rm div}(\dot{U}\nabla \dot{U})d\xi\\
&=2(pn)^{2}\int_{B^{n}}|\nabla \dot{U}|^{2}d\xi.
\end{split}
\end{equation}
Substituting \eqref{iu0y} into \eqref{bH0}, hence the validity of \eqref{bH0}  is to prove
\begin{equation}
\begin{split}
\label{bH80}
& \frac{(p-2-n)}{n\omega_{n}}\left( \int_{\sn} \varphi^{p}  d \mathcal{H}^{n-1} \right)^{2}+ (n+2-p)\int_{\sn}\varphi^{2p}d \mathcal{H}^{n-1}\\
&<\int_{\sn}\nabla_{\sn} \varphi^{p}\cdot \nabla_{\sn} \varphi^{p}d \mathcal{H}^{n-1}+2(pn)^{2}\int_{B^{n}}|\nabla \dot{U}|^{2} d \xi.
\end{split}
\end{equation}
By direct computation, the validity of \eqref{bH0} is invariant under $\varphi^{p}\mapsto \varphi^{p}+c$ for a constant $c$. Consequently, \eqref{bH80} is equivalent to
\begin{equation}\label{hu}
\begin{array}{lr}
\int_{\sn}\varphi^{p}d\mathcal{H}^{n-1}=0 \quad \Rightarrow  \\
\quad \quad (n+2-p)\int_{\sn}\varphi^{2p}d \mathcal{H}^{n-1}< \int_{\sn}|\nabla_{\sn} \varphi^{p}|^{2}d \mathcal{H}^{n-1}+2(pn)^{2}\int_{B^{n}}|\nabla \dot{U}|^{2} d \xi.
\end{array}
\end{equation}
Now, we assume $\int_{\sn}\varphi^{p}d \mathcal{H}^{n-1}=0$. Denote by $\Delta_{\sn}$ the spherical Laplace operator on $\sn$. The first eigenvalue of $(-\Delta_{\sn})$ is $0$, and the corresponding eigenspace is constant functions. Thus, the zero-mean condition on $\varphi^{p}$ illustrates that $\varphi^{p}$ is orthogonal to such eigenspace. The second eigenvalue of $(-\Delta_{\sn})$ is $n-1$, and the corresponding eigenspace consists of the restriction of linear functions of $\rnnn$ to $\sn$, since each of them is odd and $\varphi^{p}$ is even, $\varphi^{p}$ is orthogonal to this eigenspace as well. The third eigenvalue is $2n$, then by the divergence theorem on $\sn$, one sees that
\begin{equation}
\begin{split}
\label{iu0}
\int_{\sn}|\nabla_{\sn}\varphi^{p}|^{2}d \mathcal{H}^{n-1}=-\int_{\sn}\varphi^{p} \Delta_{\sn}\varphi^{p} d \mathcal{H}^{n-1}\geq 2n\int_{\sn}\varphi^{2p}d\mathcal{H}^{n-1}.
\end{split}
\end{equation}
Substituting \eqref{iu0} into \eqref{hu}, we conclude that Theorem \ref{TLP} holds. Hence, the proof of Theorem \ref{main5} is completed.
\end{proof}

{\bf Proof of Theorem \ref{main6}}

In the case $p=0$:  similarly, on the one hand, if $\phi=constant$, the log-Brunn-Minkowski inequality for torsion \eqref{LPN2}  holds in the case of Euclidean balls. On the other hand, if $\phi\neq constant$, using Lemma \ref{LE} and Proposition \ref{TP}, we conclude that to prove \eqref{LPN2}, it suffices to verify the (strict) validity of \eqref{LP2} in the case that $\Omega$ is the unit ball $B^{n}$. In this regard, we prove the following theorem.

\begin{theo}\label{TLOG0} Suppose $n\geq 2$. Let $\phi\in C^{\infty}(\sn)$ be even with $\phi\neq constant$, then the following inequality
\begin{equation}
\begin{split}
\label{bqv}
&(n+2)\int_{\sn}\phi^{2}d \mathcal{H}^{n-1} -\frac{(n+2)}{n\omega_{n}}\left(\int_{\sn}\phi d \mathcal{H}^{n-1}\right)^{2}\\
&\quad -2 \int_{\sn }\phi^{2} \frac{\nabla \dot{U}\cdot \xi}{\dot{U}} d \mathcal{H}^{n-1}< \int_{\sn}|\nabla_{\sn} \phi|^{2}d \mathcal{H}^{n-1}
\end{split}
\end{equation}
holds, where $\dot{U}$ satisfies the following equation
\begin{equation}
\label{Utv}
\left\{
\begin{array}{lr}
\Delta \dot{U}=0, & \xi\in B^{n}, \\
\dot{U}(\xi)=\frac{1}{n}\phi,   & \xi\in  \sn.
\end{array}\right.
\end{equation}
\end{theo}
\begin{proof}
By \eqref{Utv} and divergence theorem, there is
\begin{equation*}
\begin{split}
\label{Up3}
2\int_{\sn }\phi^{2} \frac{\nabla \dot{U}\cdot \xi}{\dot{U}} d \mathcal{H}^{n-1}= 2n^{2}\int_{\sn }\dot{U}(\nabla \dot{U} \cdot \xi)d \mathcal{H}^{n-1}=2n^{2}\int_{B^{n}}|\nabla \dot{U}|^{2}d\xi.
\end{split}
\end{equation*}
Then the validity of \eqref{bqv} is to verify
\begin{equation}
\begin{split}
\label{bq00}
&(n+2)\int_{\sn}\phi^{2}d \mathcal{H}^{n-1}-\frac{(n+2)}{n\omega_{n}}\left(\int_{\sn}\phi d \mathcal{H}^{n-1}\right)^{2}\\
 &<2n^{2}\int_{B^{n}}|\nabla \dot{U}|^{2}d\xi+\int_{\sn}|\nabla_{\sn} \phi|^{2}d \mathcal{H}^{n-1}.
\end{split}
\end{equation}
By direct computation, the validity of \eqref{bqv} is invariant under $\phi\mapsto \phi +c$ for a constant $c$. Consequently, \eqref{bq00} is equivalent to
\begin{equation}\label{hu2}
\int_{\sn}\phi d\mathcal{H}^{n-1}=0\Rightarrow (n+2)\int_{\sn}\phi^{2}d \mathcal{H}^{n-1}< 2n^{2}\int_{B^{n}}|\nabla \dot{U}|^{2}d\xi+ \int_{\sn}|\nabla_{\sn} \phi|^{2}d \mathcal{H}^{n-1}.
\end{equation}
Now, suppose $\int_{\sn}\phi d \mathcal{H}^{n-1}=0$. Similarly, based on the zero-mean condition and even assumption on $\phi$, then
\begin{equation}
\begin{split}
\label{iuy0}
\int_{\sn}|\nabla_{\sn}\phi|^{2}d \mathcal{H}^{n-1}=-\int_{\sn}\phi \Delta_{\sn}\phi d \mathcal{H}^{n-1}\geq 2n\int_{\sn}\phi^{2}d\mathcal{H}^{n-1}.
\end{split}
\end{equation}
Substituting \eqref{iuy0} into \eqref{hu2}, we conclude that Theorem \ref{TLOG0} holds. Hence, Theorem \ref{main6} is completely proved.

\end{proof}

Following \cite[Lemma 3.1, Lemma 3.2]{BLZ12}, as a consequence of Theorems \ref{main5} and \ref{main6},  we obtain the following $L_{p}$-Minkowski inequalities for torsion near the unit ball with $0\leq p<1$.
\begin{coro}
Suppose $n\geq 2$. Let $\varphi\in C^{\infty}(\sn)$ be even. Then there exists a sufficiently  small $\varepsilon_{0}>0$ such that for every $s_{1},s_{2}\in [-\varepsilon_{0},\varepsilon_{0}]$, the $L_{p}$-Minkowski inequalities for torsion for $0< p<1$ hold true:
\begin{equation*}\label{CLMP}
T_{p}(\Omega_{1},\Omega_{2})^{n+2}\geq T(\Omega_{1})^{n+2-p}T(\Omega_{2})^{p},
\end{equation*}
where $\Omega_{1}$ is the convex body with the support function $h_{1}=(1+s_{1}\varphi)^{\frac{1}{p}}$ and $\Omega_{2}$ is the convex body with the support function $h_{2}=(1+s_{2}\varphi)^{\frac{1}{p}}$.
\end{coro}

\begin{coro}
Suppose $n\geq 2$. Let $\phi\in C^{\infty}(\sn)$ be even. Then there exists a sufficiently small $\varepsilon_{0}>0$ such that for every $s_{1},s_{2}\in [-\varepsilon_{0},\varepsilon_{0}]$, the log-Minkowski inequality for torsion holds true:
\begin{equation*}\label{CLM0}
\int_{\sn}\log\frac{h_{\Omega_{1}}}{h_{\Omega_{2}}}d\bar{G}^{tor}(\Omega_{2},\cdot)\geq \frac{1}{n+2}\log\frac{T(\Omega_{1})}{T(\Omega_{2})},
\end{equation*}
where $\Omega_{1}$ is the convex body with the support function $h_{1}=e^{ s_{1}\phi}$ and $\Omega_{2}$ is the convex body with the support function $h_{2}=e^{ s_{2}\phi}$.
\end{coro}

\section*{Acknowledgment}The author would like to thank Yong Huang, Mohammad N. Ivaki for their helpful and valuable comments on this work.  The author also thanks the referees for detailed reading and comments that are both  thorough and insightful.

{\bf Conflict of interest:} The author declares that there is no conflict of interest.

{\bf Data availability:} No data was used for the research described in the article.

\end{document}